\documentclass[a4paper,11pt]{article}
\usepackage{graphicx,epsfig,picins}
\usepackage{fancyhdr,fancybox}
\usepackage{indentfirst}
\usepackage{titlesec}
\usepackage{verbatim}
\usepackage[sort&compress, numbers]{natbib}
\usepackage{array,dcolumn,tabularx}
\usepackage{setspace}
\usepackage{geometry}
\usepackage{extarrows,chemarrow,xypic} % 画公式中的箭头
\usepackage[small]{caption2}
\usepackage{sectsty}

\usepackage{times}     % 包括 Times Roman + Helvetica + Courier
\usepackage{latexsym}
\usepackage{amsmath}   % AMS LaTeX宏包
\usepackage{amssymb}   % 用来排版漂亮的数学公式
\usepackage{amsbsy}
\usepackage{amsthm}
\usepackage{amsfonts}
\usepackage{mathrsfs}  % 英文花体字体
\usepackage{bm}        % 数学公式中的黑斜体
\usepackage{relsize}   % 调整公式字体大小：\mathsmaller, \mathlarger
\usepackage{hyperref}  % 生成书签

\newcommand{\paperfont}{\fontsize{11pt}{1.3\baselineskip}\selectfont}
\sectionfont{\fontsize{13}{15}\selectfont}
\begin{document}

% 生成目录
%\tableofcontents

%%%%%%%%%% 定理类环境的定义 %%%%%%%%%%
\theoremstyle{definition}
\makeatletter
\thm@headfont{\bf}
\makeatother
\newtheorem{definition}{Definition}
\newtheorem{example}{Example}
\newtheorem{theorem}{Theorem}
\newtheorem{lemma}{Lemma}
\newtheorem{corollary}{Corollary}
\newtheorem{remark}{Remark}

%%%%%%%%%% 页眉和页脚的设置 %%%%%%%%%%
\lhead{}
\rhead{}
\lfoot{}
\rfoot{}

%%%%%%%%%% 一些重定义 %%%%%%%%%%
\renewcommand{\refname}{References}
\renewcommand{\figurename}{Figure}
\renewcommand{\tablename}{Table}
\renewcommand{\proofname}{Proof}

%%%%%%%%%% 论文标题、作者等 %%%%%%%%%%
\title{\textbf{Cycle symmetries and circulation fluctuations for discrete-time and continuous-time Markov chains}}
\author{Chen Jia$^{1,2}$,\;\;\;Daquan Jiang$^{1,3}$,\;\;\;Minping Qian$^{1}$ \\
\footnotesize $^1$LMAM, School of Mathematical Sciences, Peking University, Beijing 100871, P.R. China\\
\footnotesize $^2$Beijing International Center for Mathematical Research, Beijing 100871, P.R. China\\
\footnotesize $^3$Center for Statistical Science, Peking University, Beijing 100871, P.R. China\\}
\date{}                              % 日期
\maketitle                           % 生成标题
%\tableofcontents                    % 插入目录
\thispagestyle{empty}                % 首页无页眉页脚

%%%%%%%%%% 正式使用字体 %%%%%%%%%%%
\paperfont

%%%%%%%%%% 摘要 %%%%%%%%%%
\begin{abstract}
In probability theory, equalities are much less than inequalities. In this paper, we find a series of equalities which characterize the symmetry of the forming times of a family of similar cycles for discrete-time and continuous-time Markov chains. Moreover, we use these cycle symmetries to study the circulation fluctuations for Markov chains. We prove that the empirical circulations of a family of cycles passing through a common state satisfy a large deviation principle with a rate function which has an highly non-obvious symmetry. Finally, we discuss the applications of our work in statistical physics and biochemistry. \\

\noindent % 不缩进
\textbf{Keywords}: Haldane equality, current fluctuations, fluctuation theorems, large deviations, nonequilibrium \\
\textbf{Classifications}: 60J10, 60J20, 60J27, 60J28, 60F10
\end{abstract}

%%%%%%%%%% 正文 %%%%%%%%%%
\section{Introduction}\label{introduction}
Markov chains are widely used to model various stochastic systems in physics, chemistry, biology, and engineering. The trajectory of a Markov chain constantly forms various kinds of cycles. The cycle representation theory of Markov chains \cite{qian1979decomposition, qian1981markov, minping1982circulation, qian1982circulation, qian1984circulations, kalpazidou1990asymptotic, qian1991reversibility} not only possesses rich theoretical contents, but has become a fundamental tool in dealing with nonequilibrium systems in natural sciences as well. We refer to two books \cite{jiang2004mathematical, kalpazidou2007cycle} for the theoretical contents of the cycle representation theory and refer to two papers \cite{zhang2012stochastic, ge2012stochastic} for the applications of the cycle representation theory in physics, chemistry, and biology.

The earliest theoretical result of the cycle representation theory is probably the Kolmogorov's criterion for reversibility \cite{kolmogoroff1936theorie}, which claims that a stationary Markov chain is reversible if and only if the product of transition probabilities (rates) along each cycle $c$ and that along its reversed cycle $c-$ are the same. Illuminated by the diagram method \cite{hill2012free, hill2013free} developed by Hill in his study of cycle kinetics in biochemical systems, the Qians' \cite{qian1979decomposition, qian1981markov, minping1982circulation, qian1982circulation, qian1984circulations} and Kalpazidou \cite{kalpazidou1990asymptotic, kalpazidou2007cycle} introduced the important concept of circulations for Markov chains and further enriched the cycle representation theory. Let $N^c_t$ denote the number of cycle $c$ formed by a Markov chain up to time $t$. The circulation $J^c$ of cycle $c$ is a nonnegative real number defined as the following almost sure limit:
\begin{equation}\label{circulation}
J^c = \lim_{t\rightarrow\infty}\frac{1}{t}N^c_t,\;\;\;\textrm{a.s.},
\end{equation}
which represents the number of cycle $c$ formed per unit time. It turns out that a stationary Markov chain is reversible if and only if the circulations of each cycle $c$ and its reversed cycle $c-$ are the same. This explains why the cycle representation theory is naturally related to the nonequilibrium (irreversible) phenomena in natural sciences.

Recently, biophysicists have applied the cycle representation theory to study single-molecule enzyme kinetics and found an interesting relation named as the generalized Haldane equality \cite{qian2006generalized, ge2008waiting, ge2012stochastic, ge2012multivariable}. Mathematically, each chemical reaction catalyzed by an enzyme can be modeled as a Markov chain with three states (see Section \ref{biochemistry}). Let $T^c$ be the forming time of cycle $c$, which is defined as the time required for the Markov chain to form cycle $c$ for the first time, and let $T^{c-}$ be the forming time of its reversed cycle $c-$. Qian and Xie \cite{qian2006generalized} and Ge \cite{ge2008waiting} proved that for three-state Markov chains, although the distributions of $T^c$ and $T^{c-}$ can be different, their distributions, conditional on the corresponding cycle is formed early than its reversed cycle, are the same:
\begin{equation}\label{Haldaneprep}
P(T^c\leq t|T^c<T^{c-}) = P(T^{c-}\leq t|T^{c-}<T^c).
\end{equation}
This equality, which characterizes the symmetry of the forming times of a cycle and its reversed cycle, is named as the generalized Haldane equality since it generalizes of what is known as the Haldane relation
for reversible enzyme kinetics \cite{qian2006generalized}.

Now that the generalized Haldane equality holds for three-state Markov chains, it is natural to ask whether it holds for general Markov chains. If a Markov chain has only three states, then it has only two ``effective cycles" (clockwise and counterclockwise cycles) and the generalized Haldane equality can be proved using the method of quasi-time reversal \cite{ge2008waiting, ge2012multivariable}. However, this method depends too much on the cyclic topology of three-state Markov chains and cannot be generalized to general Markov chains with a large number of ``effective cycles".

In this paper, we establish some deep properties of taboo probabilities and use them to prove the generalized Haldane equality for general discrete-time and continuous-time Markov chains with denumerable state space. We find that the generalized Haldane equality not only holds for a cycle and its reversed cycle, but also holds for a family of similar cycles, which are defined as cycles passing through the same set of states (see Definition \ref{similar}). Let $c_1,c_2,\cdots,c_r$ be a family of similar cycles, let $T_{c_1},T_{c_2},\cdots,T_{c_r}$ be their forming times, and let $T = \min\{T_{c_1},T_{c_2},\cdots,T_{c_r}\}$. In this paper, we prove that although the distributions of $T_{c_1},T_{c_2},\cdots,T_{c_r}$ can be different, their distributions, conditional on the corresponding cycle is formed earlier than any other similar cycles, are the same:
\begin{equation}
P(T_{c_1}\leq t|T=T_{c_1}) = P(T_{c_2}\leq t|T=T_{c_2}) = \cdots = P(T_{c_m}\leq t|T=T_{c_m}).
\end{equation}
This equality also shows that the forming time $T$ of two or more similar cycles is independent of which one of these cycles is formed (see Corollary \ref{independent1} and Remark \ref{independentremark}), which is another important aspect of the generalized Haldane equality. The generalized Haldane equality has many variations which are closely related. These results, which include Theorems 1-4 and Corollaries 1-6, will be collectively referred to as the generalized Haldane ``equalities" in this paper.

The generalized Haldane equalities established in this paper have wide applications. One of the most important applications of the generalized Haldane equalities is to study the circulation fluctuations for Markov chains. In recent two decade, the studies about the fluctuations for stochastic systems have become a central topic in nonequilibrium statistical physics \cite{seifert2012stochastic}. Motivated by the results of numerical simulations \cite{evans1993probability}, Gallavotti and Cohen \cite{gallavotti1995dynamical} gave the first mathematical presentation of the fluctuation theorem for a class of stationary nonequilibrium systems. They proved that under suitable assumptions, the probability distribution of the phase space contraction averaged along the trajectory satisfies a large deviation principle with a rate function which has a highly non-obvious symmetry. Since then there has been a large amount of literature exploring various kinds of generalizations of the fluctuation theorem. In recent years, physicists become increasingly concerned about the fluctuations of circulations for Markov chains \cite{seifert2012stochastic}, since the entropy production, as a central concept in nonequilibrium statistical physics, can be decomposed into different cycles where the circulations emerge naturally (see Section \ref{statisticalphysics}). The entropy production fluctuations have been studied thoroughly \cite{lebowitz1999gallavotti, jiang2003entropy, seifert2005entropy}. However, the circulation fluctuations for general Markov chains remain poorly understood up till now.

Surprisingly, the generalized Haldane equalities established in this paper can be used to study the circulation fluctuations for Markov chains. The empirical circulation $J^c_t$ of cycle $c$ is defined as
\begin{equation}
J^c_t = \frac{1}{t}N^c_t.
\end{equation}
It is easy to see that the circulations defined in \eqref{circulation} are the almost sure limits of the empirical circulations. In this paper, we prove that the empirical circulations of a family of cycles $c_1,c_2,\cdots,c_r$ passing through a common state satisfy a large deviation principle with rate $t$ and good rate function $I^{c_1,c_2,\cdots,c_r}$. Moreover, we apply the generalized Haldane equalities to prove that the rate function $I^{c_1,c_2,\cdots,c_r}$ has the following highly non-obvious symmetry: if $c_k$ and $c_l$ are similar, then
\begin{equation}
\begin{split}
& I^{c_1,c_2,\cdots,c_r}(x_1,\cdots,x_k,\cdots,x_l,\cdots,x_r) \\
&= I^{c_1,c_2,\cdots,c_r}(x_1,\cdots,x_l,\cdots,x_k,\cdots,x_r)-
\left(\log\frac{\gamma^{c_k}}{\gamma^{c_l}}\right)(x_k-x_l),
\end{split}
\end{equation}
where $\gamma^{c_k}$ and $\gamma^{c_l}$ are the strengths of $c_k$ and $c_l$, respectively (see Definition \ref{strength}). In applications, we are more concerned about the fluctuations of net circulations, where the empirical net circulation $K^c_t$ of cycle $c$ is defined as
\begin{equation}
K^c_t = J^c_t-J^{c-}_t.
\end{equation}
In this paper, we prove that the empirical net circulations of cycles $c_1,c_2,\cdots,c_r$ also satisfy a large deviation principle with rate $t$ and good rate function $I_K^{c_1,c_2,\cdots,c_r}$ which has the following symmetry:
\begin{equation}
I_K^{c_1,c_2,\cdots,c_r}(x_1,\cdots,x_k,\cdots,x_r) = I^{c_1,c_2,\cdots,c_r}(x_1,\cdots,-x_k,\cdots,x_r)-
\left(\log\frac{\gamma^{c_k}}{\gamma^{c_k-}}\right)x_k.
\end{equation}
This is actually the Gallavotti-Cohen-type fluctuation theorem of net circulations. During the proof of the above results, we also obtain other types of fluctuation theorems as by-products, including the transient fluctuation theorem, the integral fluctuation theorem, and the Lebowitz-Spohn-type fluctuation theorem. All these fluctuation theorems, together with the generalized Haldane equalities, characterize the symmetries of a family of similar cycles for Markov chains from different aspects.

At the end of this paper, we discuss the applications of our work in nonequilibrium statistical physics and biochemistry. This shows that our work would have a board application prospect in natural sciences.

\section{Rigorous definitions of cycles and their forming times}
In this section, we shall give the rigorous definitions of cycles and their forming times for discrete-time and continuous-time Markov chains.

We first give the definitions of cycles. Let $X=(X_t)_{t\geq 0}$ be a time-homogeneous discrete-time or continuous-time Markov chain with denumerable state space $\mathbb{S}$ defined on some probability space $(\Omega,\mathscr{F},P)$.

\begin{definition}\label{circuit}
Let $i_1\rightarrow i_2\rightarrow\cdots\rightarrow i_s\rightarrow i_1$ and $j_1\rightarrow j_2\rightarrow\cdots\rightarrow j_r\rightarrow j_1$ be two directed circuits on complete graph with vertex set $\mathbb{S}$. Then the two directed circuits are called equivalent if $r=s$ and if there exists $1\leq k\leq s$, such that $i_{k+1}=j_1,i_{k+2}=j_2,\cdots,i_{k+s}=j_s$, where we have used the convention that $i_{s+l}=i_l$ for each integer $l$.
\end{definition}

According to the above definition, two directed circuits are called equivalent if one can be transformed into the other by a cyclic permutation. For example, the three directed circuits, $1\rightarrow2\rightarrow3\rightarrow1$, $2\rightarrow3\rightarrow1\rightarrow2$, and $3\rightarrow1\rightarrow2\rightarrow3$, are equivalent.

\begin{definition}
Let $i_1,i_2,\cdots,i_s$ be distinct states in $\mathbb{S}$. Then the equivalence class of the directed circuit $i_1\rightarrow i_2\rightarrow\cdots\rightarrow i_s\rightarrow i_1$ under the equivalence relation described in Definition \ref{circuit} is called a cycle and is denoted by $(i_1,i_2,\cdots,i_s)$.
\end{definition}

According to the above definition, two cycles are the same if one can be transformed into the other by a cyclic permutation. For example, the three cycles, $(1,2,3)$, $(2,3,1)$, and $(3,1,2)$, represent the same cycle.

We next give the definition of the forming times of cycles for discrete-time Markov chains. Let $X=(X_n)_{n\geq 0}$ be an irreducible and recurrent discrete-time Markov chain with denumerable state space $\mathbb{S}$ and transition probability matrix $P = (p_{ij})_{i,j\in\mathbb{S}}$.

To this end, we must introduce the concept of the derived chain. It can be proved that with probability one, the trajectory of $X$ will generate an infinite sequence of cycles \cite{jiang2004mathematical}. If we discard the cycles formed by $X$ and keep track of the remaining states in the trajectory, then we obtain a new Markov chain $Y$ called the derived chain. We shall give the rigorous definitions of the derived chain later, but the basic ideas should be clear from the following example.

\begin{example}
If the trajectory of the Markov chain $X$ is $\{1,2,3,2,4,5,2,3,1,\cdots\}$, then the corresponding trajectory of the derived chain $Y$ and the cycles formed are as follows:
\begin{table}[!htb]
\centering
\begin{tabular}{|c|c|c|c|c|c|c|c|c|c|}
  \hline
  $n$             & 0   & 1     & 2       & 3     & 4       & 5         & 6       & 7       & 8       \\
  $X_n$           & 1   & 2     & 3       & 2     & 4       & 5         & 2       & 3       & 1       \\
  $Y_n$           & [1] & [1,2] & [1,2,3] & [1,2] & [1,2,4] & [1,2,4,5] & [1,2]   & [1,2,3] & [1]     \\
  cycles formed   &     &       &         & (2,3) &         &           & (2,4,5) &         & (1,2,3) \\
  \hline
\end{tabular}
\end{table}
\end{example}

In order to give the rigorous definitions of the derived chain, we introduce several notations. We denote an finite sequence $i_1,i_2,\cdots,i_s$ of distinct states by $[i_1,i_2,\cdots,i_s]$ and denote the collection of all finite sequences of distinct states by $[\mathbb{S}]$, that is,
\begin{equation}
[\mathbb{S}] = \{[i_1,i_2,\cdots,i_s]: \textrm{$s\geq 1$, $i_1,\cdots,i_s$ are distinct states in $\mathbb{S}$}\}.
\end{equation}
We also define a map $\{\cdot,\cdot\}$ from $[\mathbb{S}]\times \mathbb{S}$ into $[\mathbb{S}]$ by
\begin{equation}
\{[i_1,i_2,\cdots,i_s],i\} =
\begin{cases}
[i_1,i_2,\cdots,i_s,i],~~~\textrm{if}~i\notin\{i_1,i_2,\cdots,i_s\} \\
[i_1,i_2,\cdots,i_k],~~~\textrm{if}~i=i_k~\textrm{for some}~1\leq k\leq s.
\end{cases}
\end{equation}

\begin{definition}
The derived chain $Y=(Y_n)_{n\geq 0}$ of $X$ is defined as $Y_0 = [X_0]$ and $Y_n = \{Y_{n-1},X_n\}$ for each $n\geq 1$.
\end{definition}

It can be proved that the derived chain $Y$ is a time-homogeneous Markov chain with denumerable state space $[\mathbb{S}]$ \cite{jiang2004mathematical}.

\begin{definition}
Let $c=(i_1,i_2,\cdots,i_s)$ be a cycle. For each $\omega\in\Omega$, we say that the trajectory $X(\omega)$ forms cycle $c$ at time $n$ if there exists $1\leq k\leq s$ and distinct states $j_1,j_2,\cdots,j_r\notin\{i_1,i_2,\cdots,i_s\}$ such that $Y_{n-1}(\omega) = [j_1,j_2,\cdots,j_r,i_k,i_{k+1},\cdots,i_{k+s-1}]$ and $Y_n(\omega) = [j_1,j_2,\cdots,j_r,i_k]$, where we have used the convention that $i_{s+l}=i_l$ for each integer $l$.
\end{definition}

\begin{definition}
Let $c$ be a cycle. Then the forming time $T^c$ of cycle $c$ by $X$ is defined as
\begin{equation}
T^c(\omega) = \inf\{n\geq 1: \textrm{the trajectory $X(\omega)$ forms cycle $c$ at time $n$}\}.
\end{equation}
For each $m\geq 1$, the $m$-th forming time $T^c_m$ of cycle $c$ can be defined inductively as the forming time of cycle $c$ by the Markov chain $(X_{T^c_{m-1}+n})_{n\geq 0}$, where $T^c_0$ is understood as $0$.
\end{definition}

\begin{lemma}\label{finiteness1}
Let $c=(i_1,i_2,\cdots,i_s)$ be a cycle. \\
(i) If $p_{i_1i_2}p_{i_2i_3}\cdots p_{i_si_1}>0$, then $T^c_m<\infty$ almost surely for each $m\geq 1$.\\
(ii) If $p_{i_1i_2}p_{i_2i_3}\cdots p_{i_si_1}=0$, then $T^c_m=\infty$ almost surely for each $m\geq 1$.
\end{lemma}

\begin{proof}
It is easy to see that (ii) holds. We next prove (i). To this end, we only need to prove that $T^c<\infty$ almost surely. Let $H=\{i_1,i_2,\cdots,i_s\}$. Let $\tau_0 = \inf\{n\geq 0: X_n\in H\}$ and let $\tau_m = \inf\{n>\tau_{m-1}: X_n = X_{\tau_0}\}$ for each $m\geq 1$. Since $X$ is recurrent, it is easy to see that $\tau_m<\infty$ almost surely for each $m$. Let
\begin{equation}
N(\omega) = \inf\{m\geq 1: \textrm{the trajectory $X(\omega)$ forms cycle $c$ at time $\tau_m$}\}.
\end{equation}
We now fix some $1\leq k\leq s$. By the strong Markov property, conditional on $\{X_{\tau_0}=i_k\}$, $N$ follows a geometric distribution with parameter $p_k\geq p_{i_1i_2}p_{i_2i_3}\cdots p_{i_si_1}>0$. This shows that $N<\infty$ almost surely conditional on $\{X_{\tau_0}=i_k\}$. By the arbitrariness of $k$, we obtain that $N<\infty$ almost surely. This implies that for almost every $\omega$, the trajectory $X(\omega)$ will form cycle $c$ in finite time, that is, $T^c<\infty$ almost surely.
\end{proof}

We finally give the definition of the forming times of cycles for continuous-time Markov chains. Let $X=(X_t)_{t\geq 0}$ be an irreducible and recurrent continuous-time Markov chain with denumerable state space $\mathbb{S}$ and transition rate matrix $Q = (q_{ij})_{i,j\in\mathbb{S}}$. Let $(J_n)_{n\geq 0}$ be the jump times of $X$ with the convention of $J_0 = 0$. For each $n\geq 0$, let $\bar{X}_n = X_{J_n}$. Then $\bar{X}=(\bar{X}_n)_{n\geq 0}$ is the embedded chain of $X$.

\begin{definition}\label{formingtime}
Let $c$ be a cycle. Let $\bar{T}^c$ be the forming time of cycle $c$ by the embedded chain $\bar{X}$. Then the forming time $T^c$ of cycle $c$ by $X$ is defined as
\begin{equation}
T^c = J_{\bar{T}^c}.
\end{equation}
For each $m\geq 1$, let $\bar{T}^c_m$ be the $m$-th forming time of cycle $c$ by the embedded chain $\bar{X}$. Then the $m$-th forming time $T^c_m$ of cycle $c$ by $X$ is defined as
\begin{equation}
T^c_m = J_{\bar{T}^c_m}.
\end{equation}
\end{definition}

\begin{lemma}\label{finiteness2}
Let $c=(i_1,i_2,\cdots,i_s)$ be a cycle. \\
(i) If $q_{i_1i_2}q_{i_2i_3}\cdots q_{i_si_1}>0$, then $T^c_m<\infty$ almost surely for each $m\geq 1$.\\
(ii) If $q_{i_1i_2}q_{i_2i_3}\cdots q_{i_si_1}=0$, then $T^c_m=\infty$ almost surely for each $m\geq 1$.
\end{lemma}

\begin{proof}
It is easy to see that (ii) holds. We next prove (i). Since $X$ is irreducible and recurrent, the embedded chain $\bar{X}$ is also irreducible and recurrent. By Lemma \ref{finiteness1}, we see that $\bar{T}^c_m<\infty$ almost surely. Since $X$ is irreducible and recurrent, it is non-explosive, which means that $J_n<\infty$ almost surely for each $n$. The above two facts show that $T^c_m = J_{\bar{T}^c_m} < \infty$ almost surely.
\end{proof}

We have defined the forming time of a particular cycle. We shall now define the forming time of two or more cycles.

\begin{definition}\label{formingtimefamily}
Let $c_1,c_2,\cdots,c_r$ be a family of cycles and let $T^{c_1},T^{c_2},\cdots,T^{c_r}$ be their forming times. Then the forming time $T$ of $c_1,c_2,\cdots,c_r$ by $X$ is defined as
\begin{equation}
T = \min\{T^{c_1},T^{c_2},\cdots,T^{c_r}\}.
\end{equation}
For each $m\geq 1$, the $m$-th forming time $T_m$ of $c_1,c_2,\cdots,c_r$ can be defined inductively as the forming time of $c_1,c_2,\cdots,c_r$ by the Markov chain $(X_{T_{m-1}+t})_{t\geq 0}$, where $T_0$ is understood as $0$.
\end{definition}

\section{Generalized Haldane equality for discrete-time Markov chains}
In this section, we shall state and prove the generalized Haldane equality for discrete-time Markov chains. Let $X=(X_n)_{n\geq 0}$ be an irreducible and recurrent discrete-time Markov chain with denumerable state space $\mathbb{S}$ and transition probability matrix $P = (p_{ij})_{i,j\in\mathbb{S}}$.

Before we state the generalized Haldane equality, we give the following definitions.

\begin{definition}
Let $i$ be a state and let $c=(i_1,i_2,\cdots,i_s)$ be a cycle. Then we say that cycle $c$ passes through state $i$ if $i\in\{i_1,i_2,\cdots,i_s\}$.
\end{definition}

\begin{definition}\label{similar}
Let $c_1=(i_1,i_2,\cdots,i_s)$ and $c_2=(j_1,j_2,\cdots,j_r)$ be two cycles. Then $c_1$ and $c_2$ are called similar if $s=r$ and $\{i_1,i_2,\cdots,i_s\}=\{j_1,j_2,\cdots,j_r\}$.
\end{definition}

According to the above two definitions, two cycles are similar if they pass through the same set of states. It is easy to see that similarity is an equivalence relation on the set of all cycles. For example, the six cycles, $c_1=(1,2,3,4)$, $c_2=(1,2,4,3)$, $c_3=(1,3,2,4)$, $c_4=(1,3,4,2)$, $c_5=(1,4,2,3)$, and $c_6=(1,4,3,2)$, are similar.

We next give the definition of the strengths of cycles for discrete-time Markov chains.

\begin{definition}
Let $c=(i_1,i_2,\cdots,i_s)$ be a cycle. Then the strength $\gamma^c$ of cycle $c$ is defined as
\begin{equation}
\gamma^c = p_{i_1i_2}p_{i_2i_3}\cdots p_{i_si_1}.
\end{equation}
\end{definition}

In the following discussion, the forming time of cycle $c$ is always denoted by $T^c$ and the strength of cycle $c$ is always denoted by $\gamma^c$ without further explanation.

The generalized Haldane equality, which characterizes the symmetry of the forming times of a family of similar cycles, is stated in the following theorem.

\begin{theorem}\label{Haldane1}
Let $c_1,c_2,\cdots,c_r$ be a family of similar cycles. Let $T = \min\{T^{c_1},T^{c_2},\cdots,T^{c_r}\}$. Then \\
(i) for each $n\geq 1$ and any $1\leq k,l \leq r$,
\begin{equation}\label{quotient1}
\frac{P(T^{c_k}=n,T=T^{c_k})}{P(T^{c_l}=n,T=T^{c_l})} = \frac{P(T=T^{c_k})}{P(T=T^{c_l})} = \frac{\gamma^{c_k}}{\gamma^{c_l}};
\end{equation}
(ii) for each $n\geq 1$,
\begin{equation}\label{conditional1}
P(T^{c_1}=n|T=T^{c_1}) = P(T^{c_2}=n|T=T^{c_2}) = \cdots = P(T^{c_r}=n|T=T^{c_r}).
\end{equation}
\end{theorem}

\begin{remark}
The above theorem, which seems a bit counter-intuitive at first sight, shows that although the distributions of the forming times of a family of similar cycles may not be the same, their distributions, conditional on the corresponding cycle is formed earlier than any other similar cycles, are the same. This is the first aspect of the generalized Haldane equality.
\end{remark}

\begin{remark}
If both the numerator and denominator in \eqref{quotient1} are 0, then \eqref{quotient1} is understood to hold trivially. In addition, if $P(T=T^{c_k})=0$ for some $k$, then \eqref{conditional1} is understood to hold trivially.
\end{remark}

In order to prove the generalized Haldane equality, we need to establish some deep properties of taboo probabilities. Let us first recall the definition of taboo probabilities, also called transition probabilities with a taboo set \cite{chung1967markov}.

\begin{definition}
Let $i,j$ be two states and $H$ be a subset of $\mathbb{S}$. Then the $n$-step transition probability from state $i$ to state $j$ with taboo set $H$ is defined as
\begin{equation}
p_{ij}^H(n) = P_i(X_n=j,X_1,\cdots,X_{n-1}\notin H).
\end{equation}
If the taboo set is the union of a set $H$ and a finite number of states $k_1,\cdots,k_s$, then we shall denote the taboo probability by $p_{ij}^{H,k_1,\cdots,k_s}(n)$.
\end{definition}

The next four lemma give some deep properties of taboo properties.

\begin{lemma}\label{basic}
Let $H$ be a subset of $\mathbb{S}$ and let $k\notin H$. Then for each $n\geq 0$ and any two states $i,j$,
\begin{equation}
p_{ij}^H(n) = p_{ij}^{H,k}(n)+\sum_{m=1}^{n-1}p_{ik}^H(m)p_{kj}^{H,k}(n-m).
\end{equation}
\end{lemma}

\begin{proof}
When $n=0$ or $n=1$, it is easy to check that the theorem holds. We next prove the theorem for $n\geq 2$. Note that
\begin{equation}
p_{ij}^H(n) = p_{ij}^{H,k}(n)+P_i(X_n=j,X_1,\cdots,X_{n-1}\notin H,k\in\{X_1,\cdots,X_{n-1}\}).
\end{equation}
Then by the Markov property, we obtain that
\begin{equation*}
\begin{split}
& P_i(X_n=j,X_1,\cdots,X_{n-1}\notin H,k\in\{X_1,\cdots,X_{n-1}\}) \\
&= \sum_{m=1}^{n-1}P_i(X_n=j,X_1,\cdots,X_{n-1}\notin H,X_m=k,X_{m+1},\cdots,X_{n-1}\neq k) \\
&= \sum_{m=1}^{n-1}P_i(X_m=k,X_1,\cdots,X_{m-1}\notin H)P_k(X_{n-m}=j,X_1,\cdots,X_{n-m-1}\notin H\cup\{k\}) \\
&= \sum_{m=1}^{n-1}p_{ik}^H(m)p_{kj}^{H,k}(n-m).
\end{split}
\end{equation*}
This completes the proof of this lemma.
\end{proof}

\begin{lemma}\label{transposition}
Let $H$ be a subset of $\mathbb{S}$. Let $i,j\notin H$ and $i\neq j$. Then for each $n\geq 0$,
\begin{equation}
\sum_{m=0}^{n}p_{ii}^H(m)p_{jj}^{H,i}(n-m) = \sum_{m=0}^{n}p_{jj}^H(m)p_{ii}^{H,j}(n-m).
\end{equation}
\end{lemma}

\begin{proof}
By Lemma \ref{basic}, we have
\begin{equation}
\begin{split}
& \sum_{m=0}^{n}p_{ii}^H(m)p_{jj}^{H,i}(n-m) \\
&= \sum_{m=0}^np_{ii}^H(m)p_{jj}^H(n-m) - \sum_{m=0}^np_{ii}^H(m)\sum_{l=1}^{n-m-1}p_{ji}^H(l)p_{ij}^{H,i}(n-m-l) \\
&= \sum_{m=0}^np_{ii}^H(m)p_{jj}^H(n-m) - \sum_{m=0}^np_{ii}^H(m)\sum_{l=0}^{n-m}p_{ji}^H(l)p_{ij}^{H,i}(n-m-l) \\
&= \sum_{m=0}^np_{ii}^H(m)p_{jj}^H(n-m) - \sum_{l=0}^np_{ji}^H(l)\sum_{m=0}^{n-l}p_{ii}^H(m)p_{ij}^{H,i}(n-m-l). \\
\end{split}
\end{equation}
Using Lemma \ref{basic} again, we have
\begin{equation}
\begin{split}
& \sum_{m=0}^{n-l}p_{ii}^H(m)p_{ij}^{H,i}(n-m-l) = p_{ij}^{H,i}(n-l) + \sum_{m=1}^{n-l-1}p_{ii}^H(m)p_{ij}^{H,i}(n-m-l) \\
&= p_{ij}^{H,i}(n-l) + p_{ij}^H(n-l) - p_{ij}^{H,i}(n-l) = p_{ij}^H(n-l).
\end{split}
\end{equation}
Thus we obtain that
\begin{equation}
\sum_{m=0}^{n}p_{ii}^H(m)p_{jj}^{H,i}(n-m) = \sum_{m=0}^np_{ii}^H(m)p_{jj}^H(n-m) - \sum_{l=0}^np_{ji}^H(l)p_{ij}^H(n-l).
\end{equation}
Commuting $i$ and $j$ in the above equation, we finally obtain that
\begin{equation}
\begin{split}
& \sum_{m=0}^{n}p_{jj}^H(m)p_{ii}^{H,j}(n-m) = \sum_{m=0}^np_{jj}^H(m)p_{ii}^H(n-m) - \sum_{l=0}^np_{ij}^H(l)p_{ji}^H(n-l) \\
&= \sum_{m=0}^np_{ii}^H(m)p_{jj}^H(n-m) - \sum_{l=0}^np_{ji}^H(l)p_{ij}^H(n-l) = \sum_{m=0}^{n}p_{ii}^H(m)p_{jj}^{H,i}(n-m),
\end{split}
\end{equation}
which gives the desired result.
\end{proof}

\begin{lemma}\label{invariance}
Let $H$ be a subset of $\mathbb{S}$. For any finite sequence $i_1,i_2,\cdots,i_s$ of distinct states, let
\begin{equation}
G^H_n(i_1,i_2,\cdots,i_s) = \sum_{n_1+n_2+\cdots+n_s=n}p_{i_1i_1}^{H}(n_1)p_{i_2i_2}^{H,i_1}(n_2)\cdots p_{i_si_s}^{H,i_1,\cdots,i_{s-1}}(n_s).
\end{equation}
Then for each $n\geq 0$, $G^H_n(i_1,i_2,\cdots,i_s)$ is invariant under any permutation of $i_1,\cdots,i_s$.
\end{lemma}

\begin{proof}
Since any permutation can be decomposed into the product of some transpositions of adjacent elements, we only need to prove that $G^H_n(i_1,i_2,\cdots,i_s)$ is invariant if we exchange two adjacent elements, $i_k$ and $i_{k+1}$, and keep all other elements fixed. By Lemma \ref{transposition}, we obtain that
\begin{eqnarray*}
&& G^H_n(i_1,\cdots,i_k,i_{k+1},\cdots,i_s) \\
&=& \sum_{n_1+\cdots+n_s=n}p_{i_1i_1}^{H}(n_1)\cdots p_{i_ki_k}^{H,i_1,\cdots,i_{k-1}}(n_k)p_{i_{k+1}i_{k+1}}^{H,i_1,\cdots,i_k}(n_{k+1})\cdots p_{i_si_s}^{H,i_1,\cdots,i_{s-1}}(n_s) \\
&=& \sum_{m=0}^n\sum_{n_1+\cdots+n_{k-1}+n_{k+2}+\cdots+n_s=n-m}p_{i_1i_1}^{H}(n_1)\cdots p_{i_{k-1}i_{k-1}}^{H,i_1,\cdots,i_{k-2}}(n_{k-1}) \\
&& p_{i_{k+2}i_{k+2}}^{H,i_1,\cdots,i_{k+1}}(n_{k+2})\cdots p_{i_si_s}^{H,i_1,\cdots,i_{s-1}}(n_s)\sum_{n_k+n_{k+1}=m}p_{i_ki_k}^{H,i_1,\cdots,i_{k-1}}(n_k)p_{i_{k+1}i_{k+1}}^{H,i_1,\cdots,i_k}(n_{k+1}) \\
&=& \sum_{m=0}^n\sum_{n_1+\cdots+n_{k-1}+n_{k+2}+\cdots+n_s=n-m}p_{i_1i_1}^{H}(n_1)\cdots p_{i_{k-1}i_{k-1}}^{H,i_1,\cdots,i_{k-2}}(n_{k-1}) \\
&& p_{i_{k+2}i_{k+2}}^{H,i_1,\cdots,i_{k+1}}(n_{k+2})\cdots p_{i_si_s}^{H,i_1,\cdots,i_{s-1}}(n_s)\sum_{n_k+n_{k+1}=m}p_{i_{k+1}i_{k+1}}^{H,i_1,\cdots,i_{k-1}}(n_k)p_{i_ki_k}^{H,i_1,\cdots,i_{k-1},i_{k+1}}(n_{k+1}) \\
&=& \sum_{n_1+\cdots+n_s=n}p_{i_1i_1}^{H}(n_1)\cdots p_{i_{k-1}i_{k-1}}^{H,i_1,\cdots,i_{k-2}}(n_{k-1}) p_{i_{k+1}i_{k+1}}^{H,i_1,\cdots,i_{k-1}}(n_k)p_{i_ki_k}^{H,i_1,\cdots,i_{k-1},i_{k+1}}(n_{k+1}) \\
&& p_{i_{k+2}i_{k+2}}^{H,i_1,\cdots,i_{k+1}}(n_{k+2})\cdots p_{i_si_s}^{H,i_1,\cdots,i_{s-1}}(n_s) \\
&=& G^H_n(i_1,\cdots,i_{k-1},i_{k+1},i_k,i_{k+2},\cdots,i_s).
\end{eqnarray*}
This completes the proof of this lemma.
\end{proof}

The following lemma will play a key role in the proof of the generalized Haldane equality.

\begin{lemma}\label{decomposition}
Let $c_1,c_2,\cdots,c_r$ be a family of cycles passing through a common state $i$. Let $T = \min\{T^{c_1},$ $T^{c_2},\cdots,T^{c_r}\}$. Let $c_k = (i,i^k_2,\cdots,i^k_s)$. Then for each $n\geq 1$,
\begin{equation}
P_i(T^{c_k}=n,T=T^{c_k}) = F^i_n(i^k_2,\cdots,i^k_s)\gamma^{c_k},
\end{equation}
where $F^i_n(i^k_2,\cdots,i^k_s)$, which is defined in \eqref{complex}, is invariant under any permutation of $i^k_2,\cdots,i^k_s$.
\end{lemma}

\begin{proof}
Note that the event $\{T^{c_k}=n,T=T^{c_k}\}$ is equivalent to saying that $X$ forms cycle $c_k$ at time $n$ and does not form cycles $c_1,c_2,\cdots,c_r$ before time $n$. In order to make this event occur, the Markov chain $X$ must finish the following procedures.

First, $X$ must take $n_1$ steps to return from $i$ to $i$ without forming cycles $c_1,c_2,\cdots,c_r$, and then jump from $i$ to $i^k_2$. Second, $X$ must take $n_2$ steps to return from $i^k_2$ to $i^k_2$ without entering $i$ and without forming cycles $c_1,c_2,\cdots,c_r$, and then jump from $i^k_2$ to $i^k_3$. Third, $X$ must take $n_3$ steps to return from $i^k_3$ to $i^k_3$ without entering $i,i^k_2$ and without forming cycles $c_1,c_2,\cdots,c_r$, and then jump from $i^k_3$ to $i^k_4$, and so on. Finally, $X$ must take $n_s$ steps to return from $i^k_s$ to $i^k_s$ without entering $i,i^k_1,\cdots,i^k_{s-1}$ and without forming cycles $c_1,c_2,\cdots,c_r$, and then jump from $i^k_s$ to $i$. Here, the steps $n_1,n_2,\cdots,n_s$ must satisfy $(n_1+1)+(n_2+1)+\cdots+(n_s+1) = n$, that is, $n_1+n_2+\cdots+n_s = n-s$.

We make a crucial observation that if $X$ does not enter $i$, it will never form any one of the cycles $c_1,c_2,\cdots,c_r$ since all these cycles pass through $i$. Let $p_{ii}^{c_1,c_2,\cdots,c_r}(n_1)$ denote the probability that $X$ takes $n_1$ steps to return from $i$ to $i$ without forming cycles $c_1,c_2,\cdots,c_r$. According to the above discussion, we obtain that
\begin{equation*}\label{first}
\begin{split}
& P_i(T^{c_k}=n,T=T^{c_k}) \\
&= \sum_{n_1+n_2+\cdots+n_s=n-s}p_{ii}^{c_1,c_2,\cdots,c_r}(n_1) p_{ii^k_2}p_{i^k_2i^k_2}^i(n_2)p_{i^k_2i^k_3}p_{i^k_3i^k_3}^{i,i^k_2}(n_3)p_{i^k_3i^k_4}\cdots p_{i^k_si^k_s}^{i,i^k_1,\cdots,i^k_{s-1}}(n_s)p_{i^k_si} \\
&= \left[\sum_{n_1=0}^{n-s}p_{ii}^{c_1,c_2,\cdots,c_r}(n_1)G^i_{n-n_1-s}(i^k_2,\cdots,i^k_s)\right] p_{ii^k_2}p_{i^k_2i^k_3}\cdots p_{i^k_si},
\end{split}
\end{equation*}
where
\begin{equation}
\begin{split}
G^i_{n-n_1-s}(i^k_2,\cdots,i^k_s) = \sum_{n_2+\cdots+n_s=n-n_1-s}p_{i^k_2i^k_2}^i(n_2)p_{i^k_3i^k_3}^{i,i^k_2}(n_3)\cdots p_{i^k_si^k_s}^{i,i^k_1,\cdots,i^k_{s-1}}(n_s).
\end{split}
\end{equation}
By Lemma \ref{invariance}, $G^i_{n-n_1-s}(i^k_2,\cdots,i^k_s)$ is invariant under any permutation of $i^k_2,\cdots,i^k_s$. Let
\begin{equation}\label{complex}
F^i_n(i^k_2,\cdots,i^k_s) = \sum_{n_1=0}^{n-s}p_{ii}^{c_1,c_2,\cdots,c_r}(n_1)G^i_{n-n_1-s}(i^k_2,\cdots,i^k_s).
\end{equation}
Then $F^i_n(i^k_2,\cdots,i^k_s)$ is invariant under any permutation of $i^k_2,\cdots,i^k_s$. This completes the proof of this lemma.
\end{proof}

\begin{remark}
The core idea in the above proof is to decompose the state transitions of each trajectory in the event $\{T^{c_k}=n,T=T^{c_k}\}$ into invalid transitions and valid transitions. During the invalid transitions, $X$ will walk around in circles without contributing to the forming of cycle $c_k$. During the valid transitions, however, $X$ will jump along cycle $c_k$. In this way, we can decompose the probability $P_i(T^{c_k}=n,T=T^{c_k})$ into the product of an invalid part $F^i_n(i^k_2,\cdots,i^k_s)$ and a valid part $\gamma^{c_k}$. The invalid part is invariant under any permutation of $i^k_2,\cdots,i^k_s$ and the valid part is independent of time $n$.
\end{remark}

We are now in a position to prove the generalized Haldane equality.

\begin{proof}[Proof of Theorem \ref{Haldane1}]
It is easy to see that (ii) is a direct corollary of (i). Thus we only need to prove (i). Since $c_1,c_2,\cdots,c_r$ are similar, they must pass through the same set of states, denoted by $H=\{i_1,i_2,\cdots,i_s\}$.

We first prove (i) when $X$ starts from a particular state $i\in H$. Write $c_k=(i,i^k_2,\cdots,i^k_s)$ and $c_l=(i,i^l_2,\cdots,i^l_s)$. By Lemma \ref{decomposition}, we have
\begin{equation}
\begin{split}
P_i(T^{c_k}=n,T=T^{c_k}) &= F^i_n(i^k_2,\cdots,i^k_s)\gamma^{c_k}, \\
P_i(T^{c_l}=n,T=T^{c_l}) &= F^i_n(i^l_2,\cdots,i^l_s)\gamma^{c_l},
\end{split}
\end{equation}
where $F^i_n(i^k_2,\cdots,i^k_s)$ is invariant under any permutation of $i^k_2,\cdots,i^k_s$. Since $c_k$ and $c_l$ are similar, $i^k_2,\cdots,i^k_s$ can be transformed into $i^l_2,\cdots,i^l_s$ by a permutation. This shows that
\begin{equation}\label{special}
\frac{P_i(T^{c_k}=n,T=T^{c_k})}{P_i(T^{c_l}=n,T=T^{c_l})} = \frac{\gamma^{c_k}}{\gamma^{c_l}}.
\end{equation}

We next prove (i) when $X$ starts from any initial distribution $\pi=(\pi_i)_{i\in\mathbb{S}}$. Let $\tau = \inf\{n\geq 0: X_n\in H\}$. It is easy to see that
\begin{eqnarray*}
&& P(T^{c_k}=n,T=T^{c_k}) \\
&=& \sum_{m=0}^nP(T^{c_k}=n,T=T^{c_k},\tau=m) \\
&=& P(T^{c_k}=n,T=T^{c_k},\tau=0)+\sum_{m=1}^nP(T^{c_k}=n,T=T^{c_k},\tau=m) \\
&=& \sum_{i\in H}\pi_i P_i(T^{c_k}=n,T=T^{c_k}) \\
&& +\sum_{m=1}^n\sum_{i\notin H}\sum_{j\in H}\pi_i P_i(T^{c_k}=n,T=T^{c_k},X_m=j,X_1,\cdots,X_{m-1}\notin H).
\end{eqnarray*}
By the Markov property, we have
\begin{equation}
\begin{split}
& P_i(T^{c_k}=n,T=T^{c_k},X_m=j,X_1,\cdots,X_{m-1}\notin H) \\
&= P_i(X_m=j,X_1,\cdots,X_{m-1}\notin H)P_j(T^{c_k}=n-m,T=T^{c_k}) \\
&= p_{ij}^H(m)P_j(T^{c_k}=n-m,T=T^{c_k}).
\end{split}
\end{equation}
Thus we obtain that
\begin{eqnarray*}\label{generalinitial}
P(T^{c_k}=n,T=T^{c_k}) &=& \sum_{i\in H}\pi_i P_i(T^{c_k}=n,T=T^{c_k}) \\
&& +\sum_{m=1}^n\sum_{i\notin H}\sum_{j\in H}\pi_ip_{ij}^H(m)P_j(T^{c_k}=n-m,T=T^{c_k}).
\end{eqnarray*}
According to \eqref{special} and the above equation, we see that
\begin{equation}
\frac{P(T^{c_k}=n,T=T^{c_k})}{P(T^{c_l}=n,T=T^{c_l})} = \frac{\gamma^{c_k}}{\gamma^{c_l}}.
\end{equation}
Since the above equation holds for each $n$, we obtain the desired result.
\end{proof}

The next corollary gives another aspect of the generalized Haldane equality.

\begin{corollary}\label{independent1}
Let $c_1,c_2,\cdots,c_r$ be a family of similar cycles. Let $T = \min\{T^{c_1},T^{c_2},\cdots,T^{c_r}\}$. Then for each $n\geq 0$ and each $1\leq k\leq r$,
\begin{equation}
P(T=n,T=T^{c_k}) = P(T=n)P(T=T^{c_k}).
\end{equation}
\end{corollary}

\begin{proof}
By Theorem \ref{Haldane1}, the probability $P(T=n|T=T^{c_k})$ is the same for each $k$. This implies that for each $k$,
\begin{equation}
P(T=n|T=T^{c_k}) = P(T=n).
\end{equation}
Thus we obtain that
\begin{equation}
P(T=n,T=T^{c_k}) = P(T=n|T=T^{c_k})P(T=T^{c_k}) = P(T=n)P(T=T^{c_k}),
\end{equation}
which gives the desired result.
\end{proof}

\begin{remark}\label{independentremark}
The notations are the same as in Corollary \ref{independent1}. Let $\xi$ be a random variable defined by
\begin{equation}
\xi =
\begin{cases}
c_1, &\textrm{if the trejectory of $X$ forms cycle $c_1$ at time $T$}, \\
c_2, &\textrm{if the trejectory of $X$ forms cycle $c_2$ at time $T$}, \\
\cdots, \\
c_r, &\textrm{if the trejectory of $X$ forms cycle $c_r$ at time $T$}.
\end{cases}
\end{equation}
Then Corollary \ref{independent1} shows that $T$ and $\xi$ are independent. This suggests that the forming time of two or more similar cycles is independent of which one of these cycles is formed. This is another important aspect of the generalized Haldane equality.
\end{remark}

In applications, we are more concerned about the symmetry of a cycle and its reversed cycle. Thus we give the following definition.

\begin{definition}
Let $c=(i_1,i_2,\cdots,i_s)$ be a cycle. Then the reversed cycle $c-$ of cycle $c$ is defined as $c-=(i_1,i_s,\cdots,i_2)$. The cycles $c$ and $c-$ are called conjugate.
\end{definition}

For example, the two cycles $c=(1,2,3)$ and $c-=(1,3,2)$ are conjugate. It is easy to see that conjugate cycles must be similar. Now that the generalized Haldane equality holds for similar cycles, it also holds for conjugate cycles. Thus we obtain the following corollary.

\begin{corollary}
Let $c=(i_1,i_2,\cdots,i_s)$ be a cycle. Then \\
(i) for each $n\geq 0$,
\begin{equation}
\frac{P(T^c=n,T^c<T^{c-})}{P(T^{c-}=n,T^{c-}<T^c)} = \frac{P(T^c<T^{c-})}{P(T^{c-}<T^c)} = \frac{p_{i_1i_2}p_{i_2i_3}\cdots p_{i_si_1}}{p_{i_1i_s}p_{i_si_{s-1}}\cdots p_{i_2i_1}};
\end{equation}
(ii) for each $n\geq 0$,
\begin{equation}
P(T^c=n|T^c<T^{c-}) = P(T^{c-}=n|T^{c-}<T^c);
\end{equation}
(iii) for each $n\geq 0$,
\begin{equation}
P(T^c\wedge T^{c-}=n,T^c<T^{c-}) = P(T^c\wedge T^{c-}=n)P(T^c<T^{c-}).
\end{equation}
\end{corollary}

\begin{proof}
This corollary follows directly from Theorem \ref{Haldane1} and Corollary \ref{independent1}.
\end{proof}

\begin{remark}
The above corollary generalizes the so-called generalized Haldane equality (see \eqref{Haldaneprep} in Section \ref{introduction}) found by biophysicists in three-state Markov chains \cite{qian2006generalized, ge2008waiting, ge2012stochastic}.
\end{remark}

\section{Generalizations of the generalized Haldane equality}
We have seen that the most important intermediate step in the proof of the generalized Haldane equality is Lemma \ref{decomposition}, in which we decompose the probability $P_i(T^{c_k}=n,T=T^{c_k})$ into an invalid part and a valid part. However, we notice that the conditions stated in Lemma \ref{decomposition} are much weaker than those stated in Theorem \ref{Haldane1}. This suggests that the generalized Haldane equality can be further generalized, as stated in the following theorem.

\begin{theorem}\label{Haldane2}
Let $c_1,c_2,\cdots,c_r$ be a family of cycles passing through a common state $i$. Let $T = \min\{T^{c_1},$ $T^{c_2},\cdots,T^{c_r}\}$. Assume that $c_k$ and $c_l$ are similar for some two indices $1\leq k,l\leq r$. Then \\
(i) for each $n\geq 1$,
\begin{equation}
\frac{P_i(T^{c_k}=n,T=T^{c_k})}{P_i(T^{c_l}=n,T=T^{c_l})} = \frac{P_i(T=T^{c_k})}{P_i(T=T^{c_l})} = \frac{\gamma^{c_k}}{\gamma^{c_l}};
\end{equation}
(ii) for each $n\geq 1$,
\begin{equation}
P_i(T^{c_k}=n|T=T^{c_k}) = P_i(T^{c_l}=n|T=T^{c_l}).
\end{equation}
\end{theorem}

\begin{proof}
It is easy to see that (ii) is a direct corollary of (i). Thus we only need to prove (i). Write $c_k = (i,i^k_2,\cdots,i^k_s)$ and $c_l = (i,i^l_2,\cdots,i^l_s)$. By Lemma \ref{decomposition}, we have
\begin{equation}
\begin{split}
P_i(T^{c_k}=n,T=T^{c_k}) &= F^i_n(i^k_2,\cdots,i^k_s)\gamma^{c_k}, \\
P_i(T^{c_l}=n,T=T^{c_l}) &= F^i_n(i^l_2,\cdots,i^l_s)\gamma^{c_l},
\end{split}
\end{equation}
where $F^i_n(i^k_2,\cdots,i^k_s)$ is invariant under any permutation of $i^k_2,\cdots,i^k_s$. Since $c_k$ and $c_l$ are similar, $i^k_2,\cdots,i^k_s$ can be transformed into $i^l_2,\cdots,i^l_s$ by a permutation. This shows that
\begin{equation}
\frac{P_i(T^{c_k}=n,T=T^{c_k})}{P_i(T^{c_l}=n,T=T^{c_l})} = \frac{\gamma^{c_k}}{\gamma^{c_l}}.
\end{equation}
Since the above equation holds for each $n$, we obtain the desired result.
\end{proof}

\begin{remark}
There are two crucial differences between Theorem \ref{Haldane1} and Theorem \ref{Haldane2}. The first difference is that in Theorem \ref{Haldane1}, we require that the cycles $c_1,c_2,\cdots,c_r$ are similar, while in Theorem \ref{Haldane2}, we only require that the cycles $c_1,c_2,\cdots,c_r$ pass through a common state. The second difference is that Theorem \ref{Haldane1} holds for Markov chains starting from any initial distributions, while Theorem \ref{Haldane2} only holds for Markov chains starting from a particular state.
\end{remark}

Now that the above theorem holds for similar cycles, it also holds for conjugate cycles. Thus we obtain the following corollary.

\begin{corollary}
Let $c_1,c_2,\cdots,c_r$ be a family of cycles passing through a common state $i$. Let $T = \min\{T^{c_1},$ $T^{c_1-},\cdots,T^{c_r},T^{c_r-}\}$. Then \\
(i) for each $n\geq 1$ and each $1\leq k\leq r$,
\begin{equation}
\frac{P_i(T^{c_k}=n,T=T^{c_k})}{P_i(T^{c_k-}=n,T=T^{c_k-})} = \frac{P_i(T=T^{c_k})}{P_i(T=T^{c_k-})} = \frac{\gamma^{c_k}}{\gamma^{c_k-}};
\end{equation}
(ii) for each $n\geq 1$ and each $1\leq k\leq r$,
\begin{equation}
P_i(T^{c_k}=n|T=T^{c_k}) = P_i(T^{c_k-}=n|T=T^{c_k-}).
\end{equation}
\end{corollary}

\begin{proof}
This corollary follows directly from Theorem \ref{Haldane2}.
\end{proof}

\begin{remark}
We have seen that the generalized Haldane equality (Theorem \ref{Haldane1}) has many variations which are closely related. These results, which include Theorems 1-2 and Corollaries 1-3, will be collectively referred to as the generalized Haldane ``equalities" in the following discussion.
\end{remark}

\section{Generalized Haldane equalities for continuous-time Markov chains}
In this section, we shall state and prove the generalized Haldane equalities for continuous-time Markov chains. Let $X=(X_t)_{t\geq 0}$ be an irreducible and recurrent continuous-time Markov chain with denumerable state space $\mathbb{S}$ and transition rate matrix $Q=(q_{ij})_{i,j\in\mathbb{S}}$.

Before we state the generalized Haldane equality, we give the definition of the strengths of cycles for continuous-time Markov chains.

\begin{definition}\label{strength}
Let $c=(i_1,i_2,\cdots,i_s)$ be a cycle. Then the strength $\gamma^c$ of cycle $c$ is defined as
\begin{equation}
\gamma^c = q_{i_1i_2}q_{i_2i_3}\cdots q_{i_si_1}.
\end{equation}
\end{definition}

The generalized Haldane equality, which characterizes the symmetry of the forming times of a family of similar cycles, is stated in the following theorem.

\begin{theorem}\label{Haldanecontinuous1}
Let $c_1,c_2,\cdots,c_r$ be a family of similar cycles. Let $T = \min\{T^{c_1},T^{c_2},\cdots,T^{c_r}\}$. Then \\
(i) for each $t>0$ and any $1\leq k,l \leq r$,
\begin{equation}\label{quotient2}
\frac{P(T^{c_k}\leq t,T=T^{c_k})}{P(T^{c_l}\leq t,T=T^{c_l})} = \frac{P(T=T^{c_k})}{P(T=T^{c_l})} = \frac{\gamma^{c_k}}{\gamma^{c_l}};
\end{equation}
(ii) for each $t>0$,
\begin{equation}\label{conditional2}
P(T^{c_1}\leq t|T=T^{c_1}) = P(T^{c_2}\leq t|T=T^{c_2}) = \cdots = P(T^{c_r}\leq t|T=T^{c_r}).
\end{equation}
\end{theorem}

\begin{proof}
It is easy to see that (ii) is a direct corollary of (i). Thus we only need to prove (i). Let $t>0$ be a fixed time. For each $m\geq 1$, let
\begin{equation}
Y^m_n = X_{nt/m}.
\end{equation}
Then $Y^m = (Y^m_n)_{n\geq 0}$ is an irreducible and recurrent discrete-time Markov chain with transition probability matrix $P_m = (p_{ij}(t/m))_{i,j\in\mathbb{S}}$, where $p_{ij}(t/m) = P_i(X_{t/m}=j)$. Let $T^{m,c}$ be the forming time of cycle $c$ by $Y^m$. Let $T^m = \min\{T^{m,c_1},T^{m,c_2},\cdots,T^{m,c_r}\}$.

Since $X$ is irreducible and recurrent, it is non-explosive, which implies that $X$ can only jump finite times before time $t$. Thus when $m$ is sufficiently large, $t/m$ is less than any of the waiting times of $X$ before time $t$. This means that the occurrence of the event $\{T^{c_k}\leq t,T=T^{c_k}\}$ implies the occurrence of the event $\{T^{m,c_k}\leq m,T^m=T^{m,c_k}\}$ when $m$ is sufficiently large. Thus we obtain that
\begin{equation}\label{left}
\{T^{c_k}\leq t,T=T^{c_k}\} \subset \bigcup_{N=1}^\infty\bigcap_{m=N}^\infty\{T^{m,c_k}\leq m,T^m=T^{m,c_k}\}.
\end{equation}
Similarly, it is easy to see that the occurrence of the event $\{T^{c_k}>t\}$ implies the occurrence of the event $\{T^{m,c_k}>m\}$ when $m$ is sufficiently large and the occurrence of the event $\{T<T^{c_k}\leq t\}$ implies the occurrence of the event $\{T^m<T^{m,c_k}\leq m\}$ when $m$ is sufficiently large. Thus we obtain that
\begin{equation}
\begin{split}
& \{T^{c_k}\leq t,T=T^{c_k}\}^c = \{T^{c_k}>t\}\cup\{T<T^{c_k}\leq t\} \\
&\subset \left(\bigcup_{N=1}^\infty\bigcap_{m=N}^\infty\{T^{m,c_k}>m\}\right) \bigcup \left(\bigcup_{N=1}^\infty\bigcap_{m=N}^\infty\{T^m<T^{m,c_k}\leq m\}\right) \\
&\subset \bigcup_{N=1}^\infty\bigcap_{m=N}^\infty\{T^{m,c_k}>m\}\cup\{T^m<T^{m,c_k}\leq m\} \\
&= \bigcup_{N=1}^\infty\bigcap_{m=N}^\infty\{T^{m,c_k}\leq m,T^m=T^{m,c_k}\}^c.
\end{split}
\end{equation}
This shows that
\begin{equation}\label{right}
\bigcap_{N=1}^\infty\bigcup_{m=N}^\infty\{T^{m,c_k}\leq m,T^m=T^{m,c_k}\} \subset \{T^{c_k}\leq t,T=T^{c_k}\}.
\end{equation}
By \eqref{left} and \eqref{right}, we have
\begin{equation}
\{T^{c_k}\leq t,T=T^{c_k}\} = \lim_{m\rightarrow\infty}\{T^{m,c_k}\leq m,T^m=T^{m,c_k}\}.
\end{equation}
By the dominated convergence theorem, we obtain that
\begin{equation}
P(T^{c_k}\leq t,T=T^{c_k}) = \lim_{m\rightarrow\infty}P(T^{m,c_k}\leq m,T^m=T^{m,c_k}).
\end{equation}
Write $c_k = (i^k_1,i^k_2,\cdots,i^k_s)$ and $c_l = (i^l_1,i^l_2,\cdots,i^l_s)$. By Theorem \ref{Haldane1}, we have
\begin{equation}
\begin{split}
& \frac{P(T^{c_k}\leq t,T=T^{c_k})}{P(T^{c_l}\leq t,T=T^{c_l})} = \lim_{m\rightarrow\infty}\frac{P(T^{m,c_k}\leq m,T^m=T^{m,c_k})}{P(T^{m,c_l}\leq m,T^m=T^{m,c_l})} \\
&= \lim_{m\rightarrow\infty}\frac{p_{i^k_1i^k_2}(t/m)p_{i^k_2i^k_3}(t/m)\cdots p_{i^k_si^k_1}(t/m)}{p_{i^l_1i^l_2}(t/m)p_{i^l_2i^l_3}(t/m)\cdots p_{i^l_si^l_1}(t/m)} = \frac{q_{i^k_1i^k_2}q_{i^k_2i^k_3}\cdots q_{i^k_si^k_1}}{q_{i^l_1i^l_2}q_{i^l_2i^l_3}\cdots q_{i^l_si^l_1}}
= \frac{\gamma^{c_k}}{\gamma^{c_l}}.
\end{split}
\end{equation}
Since the above equation holds for each $t$, we obtain the desired result.
\end{proof}

Using the techniques in the proof of Theorem \ref{Haldanecontinuous1}, we can obtain the following results parallel to those for discrete-time Markov chains. The proofs of the following results are all omitted.

\begin{corollary}
Let $c_1,c_2,\cdots,c_r$ be a family of similar cycles. Let $T = \min\{T^{c_1},T^{c_2},\cdots,T^{c_r}\}$. Then for each $t>0$ and each $1\leq k\leq r$,
\begin{equation}
P(T\leq t,T=T^{c_k}) = P(T\leq t)P(T=T^{c_k}).
\end{equation}
\end{corollary}

\begin{corollary}
Let $c=(i_1,i_2,\cdots,i_s)$ be a cycle. Then \\
(i) for each $t>0$,
\begin{equation}
\frac{P(T^c\leq t,T^c<T^{c-})}{P(T^{c-}\leq t,T^{c-}<T^c)} = \frac{P(T^c<T^{c-})}{P(T^{c-}<T^c)} = \frac{q_{i_1i_2}q_{i_2i_3}\cdots q_{i_si_1}}{q_{i_1i_s}q_{i_si_{s-1}}\cdots q_{i_2i_1}};
\end{equation}
(ii) for each $t>0$,
\begin{equation}
P(T^c\leq t|T^c<T^{c-}) = P(T^{c-}\leq t|T^{c-}<T^c);
\end{equation}
(iii) for each $t>0$,
\begin{equation}
P(T^c\wedge T^{c-}\leq t,T^c<T^{c-}) = P(T^c\wedge T^{c-}\leq t)P(T^c<T^{c-}).
\end{equation}
\end{corollary}

\begin{theorem}\label{Haldanecontinuous2}
Let $c_1,c_2,\cdots,c_r$ be a family of cycles passing through a common state $i$. Let $T = \min\{T^{c_1},$ $T^{c_2},\cdots,T^{c_r}\}$. Assume that $c_k$ and $c_l$ are similar for some two indices $1\leq k,l\leq r$. Then \\
(i) for each $t>0$,
\begin{equation}
\frac{P_i(T^{c_k}\leq t,T=T^{c_k})}{P_i(T^{c_l}\leq t,T=T^{c_l})} = \frac{P_i(T=T^{c_k})}{P_i(T=T^{c_l})} = \frac{\gamma^{c_k}}{\gamma^{c_l}};
\end{equation}
(ii) for each $t>0$,
\begin{equation}
P_i(T^{c_k}\leq t|T=T^{c_k}) = P_i(T^{c_l}\leq t|T=T^{c_l}).
\end{equation}
\end{theorem}

\begin{corollary}
Let $c_1,c_2,\cdots,c_r$ be a family of cycles passing through a common state $i$. Let $T = \min\{T^{c_1},$ $T^{c_1-},\cdots,T^{c_r},T^{c_r-}\}$. Then \\
(i) for each $t>0$ and each $1\leq k\leq r$,
\begin{equation}
\frac{P_i(T^{c_k}\leq t,T=T^{c_k})}{P_i(T^{c_k-}\leq t,T=T^{c_k-})} = \frac{P_i(T=T^{c_k})}{P_i(T=T^{c_k-})} = \frac{\gamma^{c_k}}{\gamma^{c_k-}};
\end{equation}
(ii) for each $t>0$ and each $1\leq k\leq r$,
\begin{equation}
P_i(T^{c_k}\leq t|T=T^{c_k}) = P_i(T^{c_k-}\leq t|T=T^{c_k-}).
\end{equation}
\end{corollary}

\section{Large deviations and fluctuations of empirical circulations}
The generalized Haldane equalities established in this paper have wide applications. One of the most important applications of the generalized Haldane equalities is to study the circulation fluctuations for Markov chains. In this section, we shall prove that the empirical circulations of a family of cycles passing through a common state satisfy a large deviation principle with a good rate function. Particularly, we shall use the generalized Haldane equalities to prove that the rate function has a highly non-obvious symmetry, which is closely related to the Gallavotti-Cohen-type fluctuation theorem in nonequilibrium statistical physics.

\subsection{Preliminaries}
In order to establish the large deviations of empirical circulations, we need some results about the large deviations for Markov renewal processes. To avoid misunderstanding, we give the following definitions.

\begin{definition}
Let $(\mu_t)_{t>0}$ be a family of probability measures on a Polish space $E$. Then we say that $(\mu_t)_{t>0}$ satisfies a large deviation principle with rate $t$ and good rate function $I:E\rightarrow[0,\infty]$ if \\
(i) for each $\alpha\geq 0$, the level set $\{x\in E:I(x)\leq\alpha\}$ is compact in $E$; \\
(ii) for each closed set $F$ in $E$,
\begin{equation}
\limsup_{t\rightarrow\infty}\frac{1}{t}\log\mu_t(F)\leq -\inf_{x\in F}I(x);
\end{equation}
(iii) for each open set $U$ in $E$,
\begin{equation}
\liminf_{t\rightarrow\infty}\frac{1}{t}\log\mu_t(U)\geq -\inf_{x\in U}I(x).
\end{equation}
\end{definition}

\begin{definition}
Let $\xi = (\xi_n)_{n\geq 0}$ be an irreducible discrete-time Markov chain with finite state space $E$. Assume that each $x\in E$ is associated with a Borel probability measure $\phi_x$ on $(0,\infty)$. Let $(\tau_n)_{n\geq 1}$ be a sequence of positive and finite random variables such that conditional on $(\xi_n)_{n\geq 0}$, the random variables $(\tau_n)_{n\geq 1}$ are independent and have the distribution
\begin{equation}
P(\tau_n\in\cdot|(\xi_n)_{n\geq 0}) = \phi_{\xi_{n-1}}(\cdot).
\end{equation}
Then $(\xi_n,\tau_{n+1})_{n\geq 0}$ is called a Markov renewal process.
\end{definition}

The following lemma, which is due to Mariani and Zambotti \cite{mariani2012large}, shows that the empirical flow of Markov renewal processes satisfies a large deviation principle with a good rate function.

\begin{lemma}\label{LDPMRP}
Let $(\xi_n,\tau_{n+1})_{n\geq 0}$ be a Markov renewal process. Let $T_n = \sum_{k=1}^n\tau_k$ be the $n$-th jump time of the Markov renewal process. Let $N_t = \inf\{n\geq 0:T_{n+1}>t\}$ be the number of jumps of the Markov renewal process up to time $t$. Let $Q_t\in C(E\times E,[0,\infty))$ be the empirical flow of the Markov renewal process up to time $t$ defined as
\begin{equation}
Q_t(x,y) = \frac{1}{t}\sum_{n=0}^{N_t}I_{\{\xi_n=x,\xi_{n+1}=y\}}.
\end{equation}
Then the law of $Q_t$ satisfies a large deviation principle with rate $t$ and good rate function $I:C(E\times E,[0,\infty))\rightarrow[0,\infty]$. Moreover, the rate function $I$ is convex.
\end{lemma}

\begin{proof}
The proof of this theorem can be found in \cite{mariani2012large}.
\end{proof}

\subsection{Large deviations of empirical circulations}
Let $X = (X_t)_{t\geq 0}$ be an irreducible and recurrent continuous-time Markov chain with denumerable state space $\mathbb{S}$ and transition rate matrix $Q = (q_{ij})$. In this paper, we only consider the large deviations of the empirical circulations for continuous-time Markov chains. Using similar but simpler techniques, we can obtain parallel results for discrete-time Markov chains.

\begin{definition}\label{empiricalcirculation}
Let $T^c_n$ be the $n$-th forming time of cycle $c$ by $X$ (see Definition \ref{formingtime}). Let $N^c_t = \inf\{n\geq 0: T^c_{n+1}>t\}$ be the number of cycle $c$ formed by $X$ up to time $t$. Then the empirical circulation $J^c_t$ of cycle $c$ up to time $t$ is defined as
\begin{equation}
J^c_t = \frac{1}{t}N^c_t
\end{equation}
and the empirical net circulation $K^c_t$ of cycle $c$ up to time $t$ is defined as
\begin{equation}
K^c_t = J^c_t-J^{c-}_t = \frac{1}{t}(N^c_t-N^{c-}_t).
\end{equation}
\end{definition}

The Qians' \cite{qian1982circulation} proved that the empirical circulation $J^c_t$ of each cycle $c$ converges almost surely to a nonnegative real number $J^c$, which is defined as the circulation of cycle $c$. The large deviations of the empirical circulations are stated in the following theorem.

\begin{theorem}\label{LDPcirculation}
Let $c_1,c_2,\cdots,c_r$ be a family of cycles passing through a common state $i$. Then under $P_i$, the law of $(J^{c_1}_t,J^{c_2}_t,\cdots,J^{c_r}_t)$ satisfies a large deviation principle with rate $t$ and good rate function $I^{c_1,c_2,\cdots,c_r}:\mathbb{R}^r\rightarrow[0,\infty]$.
\end{theorem}

\begin{remark}
In general, it is very difficult to obtain an explicit and computable expression of the rate function $I^{c_1,c_2,\cdots,c_r}$. However, we can use the generalized Haldane equalities established in this paper to prove that the rate function $I^{c_1,c_2,\cdots,c_r}$ has a highly non-obvious symmetry, whose specific form is given in Theorem \ref{symmetry}.
\end{remark}

If we only focus on the forming of cycles, instead of the specific state transitions, then the corresponding process is a Markov renewal process, as stated in the following lemma.

\begin{lemma}\label{MRPcycle}
Let $c_1,c_2,\cdots,c_r$ be a family of cycles passing through a common state $i$ and assume that $\gamma^{c_k}>0$ for some $1\leq k\leq r$. Let $T_n$ be the $n$-th forming time of $c_1,c_2,\cdots,c_r$ by $X$ (see Definition \ref{formingtimefamily}). Let $\tau_n = T_n-T_{n-1}$. Let $\xi_n$ be a random variable defined as
\begin{equation}\label{xi}
\begin{split}
\xi_n =
\begin{cases}
c_1, &\textrm{if the trajetory of $X$ forms cycle $c_1$ at time $T_n$}, \\
c_2, &\textrm{if the trajetory of $X$ forms cycle $c_2$ at time $T_n$}, \\
\cdots, \\
c_r, &\textrm{if the trajetory of $X$ forms cycle $c_r$ at time $T_n$}. \\
\end{cases}
\end{split}
\end{equation}
Then under $P_i$, $(\xi_n,\tau_n)_{n\geq 1}$ is a Markov renewal process.
\end{lemma}

\begin{proof}
Since $X$ starts from $i$ and $c_1,c_2,\cdots,c_r$ pass through $i$, it is easy to see that $X_{T_n} = i$ for each $n$. By the strong Markov property, the random sequence $(\xi_n,\tau_n)_{n\geq 1}$ is independent and identically distributed. This shows that $(\xi_n)_{n\geq 1}$ is a Markov chain with state space $E = \{c_1,c_2,\cdots,c_r\}$. Note that $\gamma^{c_k}>0$ for some $k$. By Lemma \ref{finiteness2}, we have $T_n\leq T^{c_k}_n<\infty$ almost surely for each $n$. This shows that $(\tau_n)_{n\geq 1}$ is a sequence of positive and finite random variables.

Since $(\xi_n,\tau_n)_{n\geq 1}$ is independent and identically distributed, for any bounded measurable function $f_1,\cdots,f_n$ on $(0,\infty)$, it is easy to see that
\begin{equation}
E_i(f_1(\tau_1)\cdots f_n(\tau_n)|(\xi_n)_{n\geq 1}) = E_i(f_1(\tau_1)|(\xi_n)_{n\geq 1})\cdots E_i(f_n(\tau_n)|(\xi_n)_{n\geq 1}).
\end{equation}
Moreover, for any Borel set $A$ in $(0,\infty)$,
\begin{equation}
P_i(\tau_n\in A|(\xi_n)_{n\geq 1}) = P_i(\tau_n\in A|\xi_n) = P_i(\tau_1\in A|\xi_1=x)|_{x=\xi_n} = \phi_{\xi_n}(A).
\end{equation}
where $\phi_x(A)=P_i(\tau_1\in A|\xi_1=x)$. The above two equations show each $x\in E$ is associated with a Borel probability measure $\phi_x$ on $(0,\infty)$ and conditional on $(\xi_n)_{n\geq 1}$, the random variables $(\tau_n)_{n\geq 1}$ are independent and have the distribution $P_i(\tau_n\in\cdot|(\xi_n)_{n\geq 1}) = \phi_{\xi_n}(\cdot)$. This shows that $(\xi_n,\tau_n)_{n\geq 1}$ is a Markov renewal process.
\end{proof}

We are now in a position to establish the large deviations of the empirical circulations.

\begin{proof}[Proof of Theorem \ref{LDPcirculation}]
We only need to prove this theorem when $\gamma^{c_k}>0$ for some $k$. Otherwise, we have $\gamma^{c_k}=0$ for each $k$. By Lemma \ref{finiteness2}, we see that $T^{c_k}=\infty$ almost surely for each $k$. This shows that $J^{c_k}_t=0$ almost surely for each $k$. In this case, the result of this theorem holds trivially.

We next assume that $\gamma^{c_k}>0$ for some $k$. By Lemma \ref{MRPcycle}, we see that $(\xi_n,\tau_n)_{n\geq 1}$ is a Markov renewal process with state space $E = \{c_1,c_2,\cdots,c_r\}$. Let $N_t = \inf\{n\geq 0: T_{n+1}>t\}$ be the number of jumps of the Markov renewal process up to time $t$. Let $Q_t\in C(E\times E,[0,\infty))$ be the empirical flow of the Markov renewal process up to time $t$ defined as
\begin{equation}
Q_t(x,y) = \frac{1}{t}\sum_{n=1}^{N_t}I_{\{\xi_n=x,\xi_{n+1}=y\}}.
\end{equation}
Note that for each $k$,
\begin{equation}
J^{c_k}_t = \frac{1}{t}N^{c_k}_t = \frac{1}{t}\sum_{n=1}^{N_t}I_{\{\xi_n=c_k\}} = \sum_{y\in E}Q_t(c_k,y).
\end{equation}
We define a continuous map $F:C(E\times E,[0,\infty))\rightarrow\mathbb{R}^r$ as
\begin{equation}\label{contraction}
F(Q) = \left(\sum_{y\in E}Q(c_1,y),\cdots,\sum_{y\in E}Q(c_r,y)\right).
\end{equation}
Thus we have
\begin{equation}
(J^{c_1}_t,\cdots,J^{c_r}_t) = F(Q_t).
\end{equation}
By Lemma \ref{LDPMRP}, the law of $Q_t$ satisfies a large deviation principle with rate $t$ and good rate function $I:C(E\times E,[0,\infty))\rightarrow[0,\infty]$. Using the contraction principle, we see that the law of $(J^{c_1}_t,\cdots,J^{c_r}_t)$ satisfies a large deviation principle with rate $t$ and good rate function $I^{c_1,\cdots,c_r}:\mathbb{R}^r\rightarrow[0,\infty]$ which can be represented as
\begin{equation}\label{ratefunction}
I^{c_1,\cdots,c_r}(x) = \inf_{Q\in F^{-1}(x)}I(Q).
\end{equation}
This completes the proof of this theorem.
\end{proof}

\subsection{Circulation fluctuations for Markov chains}
We have proved that the empirical circulations of a family of cycles $c_1,c_2,\cdots,c_r$ passing through a common state satisfy a large deviation principle with rate $t$ and good rate function $I^{c_1,c_2,\cdots,c_r}$. In this section, we shall study the properties of the rate function $I^{c_1,c_2,\cdots,c_r}$. In fact, we can prove that the rate function $I^{c_1,c_2,\cdots,c_r}$ has a highly non-obvious symmetry, whose specific form is given in the next theorem.

\begin{theorem}\label{symmetry}
The notations are the same as in Theorem \ref{LDPcirculation}. Assume that $c_k$ and $c_l$ are similar for some two indices $1\leq k,l\leq r$. Then the rate function $I^{c_1,c_2,\cdots,c_r}$ has the following symmetry:
\begin{equation}
\begin{split}
& I^{c_1,c_2,\cdots,c_r}(x_1,\cdots,x_k,\cdots,x_l,\cdots,x_r) \\
&= I^{c_1,c_2,\cdots,c_r}(x_1,\cdots,x_l,\cdots,x_k,\cdots,x_r)-
\left(\log\frac{\gamma^{c_k}}{\gamma^{c_l}}\right)(x_k-x_l).
\end{split}
\end{equation}
\end{theorem}

In order to prove the above theorem, we need several lemmas.

\begin{lemma}\label{convex}
The rate function $I^{c_1,c_2\cdots,c_r}$ is convex.
\end{lemma}

\begin{proof}
Note that the function $F$ defined in \eqref{contraction} is a linear function. This fact, together with \eqref{ratefunction}, shows that for any $0<\lambda<1$ and any $x,y\in\mathbb{R}^r$,
\begin{equation}
\begin{split}
& I^{c_1,\cdots,c_r}(\lambda x+(1-\lambda)y) = \inf_{Q\in F^{-1}(\lambda x+(1-\lambda)y)}I(Q) \\
&\leq \inf_{Q\in\lambda F^{-1}(x)+(1-\lambda)F^{-1}(y)}I(Q)
= \inf_{Q\in F^{-1}(x),R\in F^{-1}(y)}I(\lambda Q+(1-\lambda)R).
\end{split}
\end{equation}
By Lemma \ref{LDPMRP}, the rate function $I$ is convex. Thus we obtain that
\begin{equation}
\begin{split}
& I^{c_1,\cdots,c_r}(\lambda x+(1-\lambda)y) \leq \inf_{Q\in F^{-1}(x),R\in F^{-1}(y)}\lambda I(Q)+(1-\lambda)I(R) \\
&= \lambda\inf_{Q\in F^{-1}(x)}I(Q)+(1-\lambda)\inf_{R\in F^{-1}(y)}I(R)
= \lambda I^{c_1,\cdots,c_r}(x)+(1-\lambda)I^{c_1,\cdots,c_r}(y).
\end{split}
\end{equation}
This completes the proof of this lemma.
\end{proof}

The following lemma follows directly from the generalized Haldane equalities.

\begin{lemma}\label{Haldaneuseful}
Let $c_1,c_2,\cdots,c_r$ be a family of cycles passing through a common state $i$. Assume that $c_k$ and $c_l$ are similar for some two indices $1\leq k,l\leq r$. Let $T = \min\{T^{c_1},T^{c_2},\cdots,T^{c_r}\}$. Then for each $t>0$,
\begin{equation}
P_i(T\leq t,T=T^{c_k}) = P_i(T\leq t,T=T^{c_k}\wedge T^{c_l})P_i(T^{c_k}<T^{c_l}).
\end{equation}
\end{lemma}

\begin{proof}
By Theorem \ref{Haldanecontinuous2}(i), we have
\begin{equation}
\frac{P_i(T\leq t,T=T^{c_k})}{P_i(T\leq t,T=T^{c_l})} = \frac{P_i(T\leq t,T=T^{c_k}\wedge T^{c_l},T^{c_k}<T^{c_l})}{P_i(T\leq t,T=T^{c_k}\wedge T^{c_l},T^{c_l}<T^{c_k})} = \frac{\gamma^{c_k}}{\gamma^{c_l}}.
\end{equation}
Using Theorem \ref{Haldanecontinuous2}(i) again, we have
\begin{equation}
\frac{P_i(T^{c_k}\wedge T^{c_l}=T^{c_k})}{P_i(T^{c_k}\wedge T^{c_l}=T^{c_l})} = \frac{P_i(T^{c_k}<T^{c_l})}{P_i(T^{c_l}<T^{c_k})} = \frac{\gamma^{c_k}}{\gamma^{c_l}}.
\end{equation}
Combining the above two equations, we obtain that
\begin{equation}
P_i(T\leq t,T=T^{c_k}\wedge T^{c_l}|T^{c_k}<T^{c_l}) = P_i(T\leq t,T=T^{c_k}\wedge T^{c_l}|T^{c_l}<T^{c_k}).
\end{equation}
This implies that
\begin{equation}
P_i(T\leq t,T=T^{c_k}\wedge T^{c_l}|T^{c_k}<T^{c_l}) = P_i(T\leq t,T=T^{c_k}\wedge T^{c_l}).
\end{equation}
Thus we obtain that
\begin{equation}
\begin{split}
& P_i(T\leq t,T=T^{c_k}) = P_i(T\leq t,T=T^{c_k}\wedge T^{c_l},T^{c_k}<T^{c_l}) \\
&= P_i(T\leq t,T=T^{c_k}\wedge T^{c_l}|T^{c_k}<T^{c_l})P_i(T^{c_k}<T^{c_l}) \\
&= P_i(T\leq t,T=T^{c_k}\wedge T^{c_l})P_i(T^{c_k}<T^{c_l}),
\end{split}
\end{equation}
which gives the desired result.
\end{proof}

We next use the generalized Haldane equalities to prove that the joint distribution of the empirical circulations has a certain symmetry.

\begin{lemma}\label{distribution}
Let $c_1,c_2,\cdots,c_r$ be a family of cycles passing through a common state $i$. Assume that $c_k$ and $c_l$ are similar for some two indices $1\leq k,l\leq r$. Then for any $n_1,n_2,\cdots,n_r\in\mathbb{N}$,
\begin{equation}
\frac{P_i(N^{c_1}_t=n_1,\cdots,N^{c_k}_t=n_k,\cdots,N^{c_l}_t=n_l,\cdots,N^{c_r}_t=n_r)}
{P_i(N^{c_1}_t=n_1,\cdots,N^{c_k}_t=n_l,\cdots,N^{c_l}_t=n_k,\cdots,N^{c_r}_t=n_r)} = \left(\frac{\gamma^{c_k}}{\gamma^{c_l}}\right)^{n_k-n_l}.
\end{equation}
\end{lemma}

\begin{proof}
We only need to prove this lemma when $k = 1$ and $l = 2$. The proof of the other cases is totally the same. To simplify notations, let $N = n_1+\cdots+n_r$ and let
\begin{equation}
p = P_i(T^{c_1}<T^{c_2}),\;\;\;q = P_i(T^{c_2}<T^{c_1}).
\end{equation}
Let $T_n$ be the $n$-th forming time of $c_1,c_2,\cdots,c_r$ by $X$. Let $\tau_n = T_n-T_{n-1}$. Let $\xi_n$ be the random variable defined in \eqref{xi}. Then we have
\begin{eqnarray*}
&& P_i(N^{c_1}_t=n_1,\cdots,N^{c_r}_t=n_r) \\
&=& \sum_{A_1,\cdots,A_r}P_i(T_N\leq t<T_{N+1},\xi_m = c_1\;\textrm{for those}\;m\in A_1,\cdots,\xi_m = c_r\;\textrm{for those}\;m\in A_r),
\end{eqnarray*}
where the sequence $A_1,\cdots,A_r$ ranges over all partitions of $\{1,2,\cdots,N\}$ such that $\textrm{Card}(A_k)=n_k$ for each $1\leq k\leq r$. Then Lemma \ref{Haldaneuseful}, together with the fact that $(\xi_n,\tau_n)_{n\geq 1}$ is independent and identically distributed, shows that
\begin{eqnarray*}
&& P_i(N^{c_1}_t=n_1,N^{c_2}_t=n_2,N^{c_3}_t=n_3,\cdots,N^{c_r}_t=n_r) \\
&=& \sum_{A_1,\cdots,A_r}P_i(T_N\leq t<T_{N+1},\xi_m\in\{c_1,c_2\}\;\textrm{for those}\;m\in A_1\cup A_2, \\
&& \xi_m=c_3\;\textrm{for those}\;m\in A_3,\cdots,\xi_m = c_r\;\textrm{for those}\;m\in A_r)p^{n_1}q^{n_2} \\
&=& \sum_{B_1,\cdots,B_r}P_i(T_N\leq t<T_{N+1},\xi_m\in\{c_1,c_2\}\;\textrm{for those}\;m\in B_2, \\
&& \xi_m=c_3\;\textrm{for those}\;m\in B_3,\cdots,\xi_m = c_r\;\textrm{for those}\;m\in B_r)C_{n_1+n_2}^{n_1}p^{n_1}q^{n_2} \\
&=& P_i(N^{c_1}_t+N^{c_2}_t=n_1+n_2,N^{c_3}_t=n_3,\cdots,N^{c_r}_t=n_r)C_{n_1+n_2}^{n_1}p^{n_1}q^{n_2},
\end{eqnarray*}
where the sequence $B_2,\cdots,B_r$ ranges over all partitions of $\{1,2,\cdots,N\}$ such that $\textrm{Card}(B_2)=n_1+n_2$ and $\textrm{Card}(B_k)=n_k$ for each $3\leq k\leq r$. By Theorem \ref{Haldanecontinuous2}, it follows that
\begin{equation}
\begin{split}
& P_i(N^{c_1}_t=n_1,N^{c_2}_t=n_2,N^{c_3}_t=n_3,\cdots,N^{c_r}_t=n_r) \\
&= P_i(N^{c_1}_t+N^{c_2}_t=n_1+n_2,N^{c_3}_t=n_3,\cdots,N^{c_r}_t=n_r)C_{n_1+n_2}^{n_1}p^{n_1}q^{n_2} \\
&= P_i(N^{c_1}_t+N^{c_2}_t=n_1+n_2,N^{c_3}_t=n_3,\cdots,N^{c_r}_t=n_r) C_{n_1+n_2}^{n_1}p^{n_2}q^{n_1}\left(\frac{p}{q}\right)^{n_1-n_2} \\
&= P_i(N^{c_1}_t=n_2,N^{c_2}_t=n_1,N^{c_3}_t=n_3,\cdots,N^{c_r}_t=n_r) \left(\frac{\gamma^{c_1}}{\gamma^{c_2}}\right)^{n_1-n_2},
\end{split}
\end{equation}
which gives the desired result.
\end{proof}

The next lemma shows that the moment generating function of the empirical circulations has a certain symmetry.

\begin{lemma}\label{generating}
Let $c_1,c_2,\cdots,c_r$ be a family of cycles passing through a common state $i$. Assume that $c_k$ and $c_l$ are similar for some two indices $1\leq k,l\leq r$. Let
\begin{equation}
g_t(\lambda_1,\cdots,\lambda_r) = E_ie^{\lambda_1N^{c_1}_t+\cdots+\lambda_rN^{c_r}_t} = E_ie^{t(\lambda_1J^{c_1}_t+\cdots+\lambda_rJ^{c_r}_t)}.
\end{equation}
Then for each $t\geq 0$ and any $\lambda_1,\cdots,\lambda_r\in\mathbb{R}$,
\begin{equation}
\begin{split}
& g_t(\lambda_1,\cdots,\lambda_k,\cdots,\lambda_l,\cdots,\lambda_r) \\
&= g_t(\lambda_1,\cdots,\lambda_l-\log\frac{\gamma^{c_k}}{\gamma^{c_l}},\cdots,
\lambda_k+\log\frac{\gamma^{c_k}}{\gamma^{c_l}},\cdots,\lambda_r).
\end{split}
\end{equation}
\end{lemma}

\begin{proof}
We only need to prove this lemma when $k = 1$ and $l = 2$. The proof of the other cases is totally the same. By Lemma \ref{distribution}, we have
\begin{eqnarray*}
&& g_t(\lambda_1,\lambda_2,\lambda_3\cdots,\lambda_r) = E_ie^{\lambda_1N^{c_1}_t+\lambda_2N^{c_2}_t+\lambda_3N^{c_3}_t+\cdots+\lambda_rN^{c_r}_t} \\
&=& \sum_{n_1,\cdots,n_r\in\mathbb{N}}e^{\lambda_1n_1+\cdots+\lambda_rn_r} P_i(N^{c_1}_t=n_1,N^{c_2}_t=n_2,N^{c_3}_t=n_3,\cdots,N^{c_r}_t=n_r) \\
&=& \sum_{n_1,\cdots,n_r\in\mathbb{N}}e^{\lambda_1n_1+\cdots+\lambda_rn_r}
P_i(N^{c_1}_t=n_2,N^{c_2}_t=n_1,N^{c_3}_t=n_3,\cdots,N^{c_r}_t=n_r)
\left(\frac{\gamma^{c_1}}{\gamma^{c_2}}\right)^{n_1-n_2} \\
&=& \sum_{n_1,\cdots,n_r\in\mathbb{N}}e^{(\lambda_1+\log\frac{\gamma^{c_1}}{\gamma^{c_2}})n_1+
(\lambda_2-\log\frac{\gamma^{c_1}}{\gamma^{c_2}})n_2+\lambda_3n_3+\cdots+\lambda_rn_r}
P_i(N^{c_1}_t=n_2,N^{c_2}_t=n_1,N^{c_3}_t=n_3,\cdots,N^{c_r}_t=n_r) \\
&=& E_ie^{(\lambda_2-\log\frac{\gamma^{c_1}}{\gamma^{c_2}})N^{c_1}_t+
(\lambda_1+\log\frac{\gamma^{c_1}}{\gamma^{c_2}})N^{c_2}_t\lambda_3N^{c_3}_t+\cdots+\lambda_rN^{c_r}_t} \\
&=& g_t(\lambda_2-\log\frac{\gamma^{c_1}}{\gamma^{c_2}},\lambda_1+\log\frac{\gamma^{c_1}}{\gamma^{c_2}},
\lambda_3,\cdots,\lambda_r),
\end{eqnarray*}
which gives the desired result.
\end{proof}

We also need the following lemma, whose original form is given by Varadhan \cite{varadhan1984large}.

\begin{lemma}\label{Varadhan}
Let $(\mu_t)_{t>0}$ be a sequence of probability measures on a Polish space $E$ which satisfies a large deviation principle with rate $t$ and good rate function $I:E\rightarrow[0,\infty]$. Let $F:E\rightarrow\mathbb{R}$ be a continuous function. Assume that there exists $\gamma>1$ such that the following moment condition is satisfied:
\begin{equation}
\limsup_{t\rightarrow\infty}\frac{1}{t}\log\int_Ee^{\gamma tF(x)}d\mu_t(x) < \infty.
\end{equation}
Then
\begin{equation}
\lim_{t\rightarrow\infty}\frac{1}{t}\log\int_Ee^{tF(x)}d\mu_t(x) = \sup_{x\in E}(F(x)-I(x)).
\end{equation}
\end{lemma}

\begin{proof}
The proof of this lemma can be found in \cite{dembo1998large}.
\end{proof}

Using the above lemma, we can obtain the following result.

\begin{lemma}\label{applicationVaradhan}
The notations are the same as in Lemma \ref{generating}. Then for each $\lambda\in\mathbb{R}^r$,
\begin{equation}
\lim_{t\rightarrow\infty}\frac{1}{t}\log g_t(\lambda) = \sup_{x\in\mathbb{R}^r}\{\lambda\cdot x-I^{c_1,c_2,\cdots,c_r}(x)\}.
\end{equation}
\end{lemma}

\begin{proof}
Let $\lambda = (\lambda_1,\cdots,\lambda_r)$. By Theorem \ref{LDPcirculation}, the law of $(J^{c_1}_t,\cdots,J^{c_r}_t)$ satisfies a large deviation principle with rate $t$ and good rate function $I^{c_1,\cdots,c_r}$. By Lemma \ref{Varadhan}, the result of this lemma holds if the following moment condition is satisfied for each $\gamma>0$:
\begin{equation}
\limsup_{t\rightarrow\infty}\frac{1}{t}\log E_ie^{\gamma t(\lambda_1J^{c_1}_t+\cdots+\lambda_rJ^{c_r}_t)} < \infty.
\end{equation}
Note that
\begin{equation}
\begin{split}
& E_ie^{\gamma t(\lambda_1J^{c_1}_t+\cdots+\lambda_rJ^{c_r}_t)}
\leq E_ie^{\gamma|\lambda_1|N^{c_1}_t+\cdots+\gamma|\lambda_r|N^{c_r}_t} \\
&\leq E_ie^{\gamma\alpha(N^{c_1}_t+\cdots+N^{c_r}_t)}
= E_ie^{\gamma\alpha N_t},
\end{split}
\end{equation}
where $\alpha=\max\{|\lambda_1|,\cdots,|\lambda_r|\}$ and $N_t = \inf\{n\geq 0: T_{n+1}>t\}$ is the number of cycles $c_1,\cdots,c_r$ formed by $X$ up to time $t$. Since $X$ starts from $i$, in order to form any one of the cycles $c_1,\cdots,c_r$, $X$ must first leave state $i$. This shows that the $n$-th forming time $T_n$ of $c_1,\cdots,c_r$ by $X$ is larger than $n$ independent exponential random variables with the same rate $q_i$. where $q_i=\sum_{j\neq i}q_{ij}$. This further implies that $N_t$ is stochastically dominated by a Poisson random variable $R_t$ with parameter $q_it$. Thus we obtain that
\begin{equation}
\begin{split}
& E_ie^{\gamma\alpha N_t}
= \int_{-\infty}^\infty\gamma\alpha e^{\gamma\alpha x}P_i(N_t\geq x)dx
\leq \int_{-\infty}^\infty\gamma\alpha e^{\gamma\alpha x}P_i(R_t\geq x)dx \\
&= E_ie^{\gamma\alpha R_t} = \sum_{n=0}^\infty e^{\gamma\alpha n}\frac{(q_it)^n}{n!}e^{-q_it}
= \exp\left((e^{\gamma\alpha}-1)q_it\right).
\end{split}
\end{equation}
This shows that
\begin{equation}
\limsup_{t\rightarrow\infty}\frac{1}{t}\log E_ie^{\gamma t(\lambda_1J^{c_1}_t+\cdots+\lambda_rJ^{c_r}_t)}
\leq\limsup_{t\rightarrow\infty}\frac{1}{t}\log E_ie^{\gamma\alpha N_t} \leq (e^{\gamma\alpha}-1)q_i <\infty.
\end{equation}
This completes the proof of this lemma.
\end{proof}

\begin{remark}
Lemma \ref{applicationVaradhan} shows that
\begin{equation}
\lim_{t\rightarrow\infty}\frac{1}{t}\log g_t(\lambda) = (I^{c_1,c_2,\cdots,c_r})^*(\lambda),
\end{equation}
where $(I^{c_1,c_2,\cdots,c_r})^*$ is the Legendre-Fenchel transform of the rate function $I^{c_1,c_2,\cdots,c_r}$. Recall that the Legendre-Fenchel transform of a function $f:\mathbb{R}^r\rightarrow[-\infty,\infty]$ is a function $f^*:\mathbb{R}^r\rightarrow[-\infty,\infty]$ defined by
\begin{equation}
f^*(\lambda) = \sup_{x\in\mathbb{R}^r}\{\lambda\cdot x-F(x)\}.
\end{equation}
\end{remark}

The following lemma, which is called the Fenchel-Moreau theorem, gives the sufficient and necessary conditions under which the Legendre-Fenchel transform is an involution. Recall that a function $f:\mathbb{R}^r\rightarrow[-\infty,\infty]$ is called proper if $f(x)<\infty$ for at least one $x$ and $f(x)>-\infty$ for each $x$.

\begin{lemma}\label{involution}
Let $f:\mathbb{R}^r\rightarrow[-\infty,\infty]$ be a proper function. Then $f^{**}=f$ if and only if $f$ is convex and lower semi-continuous.
\end{lemma}

\begin{proof}
The proof of this lemma can be found in \cite{borwein2010convex}.
\end{proof}

We are now in a position to prove the symmetry of the rate function $I^{c_1,c_2,\cdots,c_r}$.

\begin{proof}[Proof of Theorem \ref{symmetry}]
We only need to prove this theorem when $k = 1$ and $l = 2$. The proof of the other cases is totally the same. By Lemma \ref{applicationVaradhan}, we have
\begin{equation}
\lim_{t\rightarrow\infty}\frac{1}{t}\log g_t(\lambda_1,\cdots,\lambda_r) = (I^{c_1,\cdots,c_r})^*(\lambda_1,\cdots,\lambda_r).
\end{equation}
By Lemma \ref{generating}, we have
\begin{equation}
\begin{split}
g_t(\lambda_1,\lambda_2,\lambda_3,\cdots,\lambda_r) = g_t(\lambda_2-\log\frac{\gamma^{c_1}}{\gamma^{c_2}},\lambda_1+\log\frac{\gamma^{c_1}}{\gamma^{c_2}},
\lambda_3,\cdots,\lambda_r).
\end{split}
\end{equation}
Combining the above two equations, we obtain that
\begin{equation}
\begin{split}
& (I^{c_1,\cdots,c_r})^*(\lambda_1,\lambda_2,\lambda_3,\cdots,\lambda_r) \\
&= (I^{c_1,\cdots,c_r})^*(\lambda_2-\log\frac{\gamma^{c_1}}{\gamma^{c_2}},\lambda_1+
\log\frac{\gamma^{c_1}}{\gamma^{c_2}},\lambda_3,\cdots,\lambda_r).
\end{split}
\end{equation}
By Theorem \ref{LDPcirculation} and Lemma \ref{convex}, $I^{c_1,\cdots,c_r}$ is a good rate function which is also convex. This shows that $I^{c_1,\cdots,c_r}$ is proper, convex, and lower semi-continuous. By Lemma \ref{involution}, we obtain that $I^{c_1,\cdots,c_r}=(I^{c_1,\cdots,c_r})^{**}$. Thus we have
\begin{eqnarray*}
&& I^{c_1,\cdots,c_r}(x_1,x_2,x_3,\cdots,x_r)
= (I^{c_1,\cdots,c_r})^{**}(x_1,x_2,x_3,\cdots,x_r) \\
&=& \sup_{\lambda_1,\cdots,\lambda_r\in\mathbb{R}}\{\lambda_1x_1+\cdots+\lambda_rx_r - (I^{c_1,\cdots,c_r})^*(\lambda_1,\lambda_2,\lambda_3,\cdots,\lambda_r)\} \\
&=& \sup_{\lambda_1,\cdots,\lambda_r\in\mathbb{R}}\{\lambda_1x_1+\cdots+\lambda_rx_r - (I^{c_1,\cdots,c_r})^*(\lambda_2-\log\frac{\gamma^{c_1}}{\gamma^{c_2}},
\lambda_1+\log\frac{\gamma^{c_1}}{\gamma^{c_2}},\lambda_3,\cdots,\lambda_r)\} \\
&=& \sup_{\lambda_1,\cdots,\lambda_r\in\mathbb{R}}\{(\lambda_1-\log\frac{\gamma^{c_1}}{\gamma^{c_2}})x_1+
(\lambda_2+\log\frac{\gamma^{c_1}}{\gamma^{c_2}})x_2+\lambda_3x_3+\cdots+\lambda_rx_r - \\
&& (I^{c_1,\cdots,c_r})^*(\lambda_2,\lambda_1,\lambda_3,\cdots,\lambda_r)\} \\
&=& I^{c_1,\cdots,c_r}(x_2,x_1,x_3,\cdots,x_r)-\left(\log\frac{\gamma^{c_1}}{\gamma^{c_2}}\right)(x_1-x_2),
\end{eqnarray*}
which gives the desired result.
\end{proof}

Now that Theorem \ref{LDPcirculation} and Theorem \ref{symmetry} hold for similar cycles, they also hold for conjugate cycles. Thus we obtain the following corollary.

\begin{corollary}\label{LDPconjugate}
Let $c_1,c_2,\cdots,c_r$ be a family of cycles passing through a common state $i$. Then under $P_i$, the law of $(J^{c_1}_t,J^{c_1-}_t,\cdots,J^{c_r}_t,J^{c_r-}_t)$ satisfies a large deviation principle with rate $t$ and good rate function $I^{c_1,c_1-,\cdots,c_r,c_r-}:\mathbb{R}^{2r}\rightarrow[0,\infty]$. Moreover, for each $1\leq k\leq r$, the rate function $I^{c_1,c_1-,\cdots,c_r,c_r-}$ has the following symmetry:
\begin{equation}
\begin{split}
& I^{c_1,c_1-,\cdots,c_r,c_r-}(x_1,y_1,\cdots,x_k,y_k,\cdots,x_r,y_r) \\
&= I^{c_1,c_1-,\cdots,c_r,c_r-}(x_1,y_1,\cdots,y_k,x_k,\cdots,x_r,y_r)-
\left(\log\frac{\gamma^{c_k}}{\gamma^{c_k-}}\right)(x_k-y_k).
\end{split}
\end{equation}
\end{corollary}

\begin{remark}
The generalized Haldane equalities characterize the symmetry of the forming times of a family of similar cycles. Theorem \ref{symmetry}, Lemma \ref{distribution}, and Lemma \ref{generating}, however, characterize the symmetry of the empirical circulations of a family of similar cycles from different aspects.
\end{remark}

\section{Applications in natural sciences}

\subsection{Applications in nonequilibrium statistical physics}\label{statisticalphysics}
Markov chains are widely used to model various kinds of stochastic systems in physics, chemistry, and biology. In nonequilibrium statistical physics, one of the most important physical quantities associated with a stochastic systems is the entropy production rate, which characterizes how much entropy is produced by the system per unit time. Several research groups \cite{lebowitz1999gallavotti, jiang2003entropy, seifert2005entropy} have studied the fluctuations of the empirical entropy production rate for stochastic systems modeled by Markov chains. Let $X=(X_t)_{t\geq 0}$ be an irreducible and recurrent continuous-time Markov chain with denumerable state space $\mathbb{S}$ and transition rate matrix $Q=(q_{ij})_{i,j\in\mathbb{S}}$. The empirical entropy production rate $W_t$ of $X$ up to time $t$ is defined as
\begin{equation}\label{entropyproduction}
\begin{split}
& W_t = \frac{1}{t}\log\frac{p_0(X_0)q_{\bar{X}_0\bar{X}_1}q_{\bar{X}_1\bar{X}_2}\cdots q_{\bar{X}_{\tilde{N}_t-1}\bar{X}_{\tilde{N}_t}}}{p_t(X_t)q_{\bar{X}_1\bar{X}_0}q_{\bar{X}_2\bar{X}_1}\cdots q_{\bar{X}_{\tilde{N}_t}\bar{X}_{\tilde{N}_t-1}}} = \frac{1}{t}\log\frac{p_0(X_0)}{p_t(X_t)} + \frac{1}{t}\sum_{i=0}^{\tilde{N}_t-1}\log\frac{q_{\bar{X}_i\bar{X}_{i+1}}}{q_{\bar{X}_{i+1}\bar{X}_i}},
\end{split}
\end{equation}
where $p_t = (p_t(i))_{i\in\mathbb{S}}$ is the distribution of $X$ at time $t$, $\bar{X} = (\bar{X}_n)_{n\geq 0}$ is the embedded chain of $X$, and $\tilde{N}_t$ is the number of jumps of $X$ up to time $t$. Physicists found that the empirical entropy production rate $W_t$ satisfies various kinds of fluctuation theorems. This discovery has been considered one of the most important results in nonequilibrium statistical physics in the last two decades.

Recently, physicists \cite{jiang2004mathematical, seifert2012stochastic} found that the empirical entropy production rate of Markov chains can be decomposed into different cycles. Specifically, the empirical entropy production rate $W_t$ can be decomposed as
\begin{equation}
W_t = \frac{1}{2}\sum_{c}K^c_t\log\frac{\gamma^c}{\gamma^{c-}} + W^r_t,
\end{equation}
where $c$ ranges over all cycles, $K^c_t$ is the empirical net circulation of cycle $c$ (see Definition \ref{empiricalcirculation}), $\gamma^c$ is the strength of cycle $c$, and the remainder $W^r_t$ collects the contributions of those state transitions that do not form a full cycle. This shows that the empirical net circulation $K^c_t$ of cycle $c$ is proportional to the entropy production rate of $X$ along cycle $c$. Thus it is nature to ask whether we can establish fluctuation theorems of the empirical net circulations for general Markov chains.

Fortunately, the generalized Haldane equalities established in this paper can be used to study the circulation fluctuations for Markov chains. To make the readers understand the relations between our work and nonequilibrium statistical physics, we briefly state various types of fluctuation theorems for the empirical net circulations.

We first give the definition of the affinities of cycles for Markov chains \cite{seifert2012stochastic}.

\begin{definition}\label{affinity}
Let $c$ be a cycle. Then the affinity $\rho^c$ of cycle $c$ is defined as
\begin{equation}
\rho^c = \log\frac{\gamma^c}{\gamma^{c-}}.
\end{equation}
\end{definition}

Theorems of the following type are called transient fluctuation theorems in nonequilibrium statistical physics. Transient fluctuation theorems give the probability ratio of observing trajectories that satisfy or violate the second law of thermodynamics.

\begin{theorem}\label{transient}
Let $c_1,c_2,\cdots,c_r$ be a family of cycles passing through a common state $i$. Then for any $n_1,n_2,\cdots,n_r\in\mathbb{Z}$,
\begin{equation}
\frac{P_i(K^{c_1}_t=n_1/t,\cdots,K^{c_k}_t=n_k/t,\cdots,K^{c_r}_t=n_r/t)}
{P_i(K^{c_1}_t=n_1/t,\cdots,K^{c_k}_t=-n_k/t,\cdots,K^{c_r}_t=n_r/t)} = e^{n_k\rho^{c_k}}.
\end{equation}
\end{theorem}

\begin{proof}
We only need to prove this theorem when $k = 1$. The proof of the other cases is totally the same. By Lemma \ref{distribution}, we have
\begin{eqnarray*}
&& P_i(K^{c_1}_t=n_1/t,\cdots,K^{c_k}_t=n_k/t,\cdots, K^{c_r}_t=n_r/t) \\
&=& P_i(N^{c_1}_t-N^{c_1-}_t=n_1,N^{c_2}_t-N^{c_2-}_t=n_2,\cdots,
N^{c_r}_t-N^{c_r-}_t=n_r) \\
&=& \sum_{l_1-m_1=n_1,\cdots,l_r-m_r=n_r}P_i(N^{c_1}_t=l_1,N^{c_1-}_t=m_1,N^{c_2}_t=l_2,N^{c_2-}_t=m_2,\cdots, \\
&& N^{c_r}_t=l_r,N^{c_r-}_t=m_r) \\
&=& \sum_{l_1-m_1=n_1,\cdots,l_r-m_r=n_r}P_i(N^{c_1}_t=m_1,N^{c_1-}_t=l_1,N^{c_2}_t=l_2,N^{c_2-}_t=m_2,\cdots, \\
&& N^{c_r}_t=l_r,N^{c_r-}_t=m_r)e^{(l_1-m_1)\rho^{c_1}} \\
&=& \sum_{l_1-m_1=-n_1,\cdots,l_r-m_r=n_r}P_i(N^{c_1}_t=l_1,N^{c_1-}_t=m_1,N^{c_2}_t=l_2,N^{c_2-}_t=m_2,\cdots,\\
&& N^{c_r}_t=l_r,N^{c_r-}_t=m_r)e^{n_1\rho^{c_1}} \\
&=& P_i(N^{c_1}_t-N^{c_1-}_t=-n_1,N^{c_2}_t-N^{c_2-}_t=n_2,\cdots,
N^{c_r}_t-N^{c_r-}_t=n_r)e^{n_1\rho^{c_1}} \\
&=& P_i(K^{c_1}_t=n_1/t,\cdots,K^{c_k}_t=-n_k/t,\cdots,K^{c_r}_t=n_r/t)e^{n_1\rho^{c_1}},
\end{eqnarray*}
which gives the desired result.
\end{proof}

Theorems of the following type are called Kurchan-Lebowitz-Spohn-type fluctuation theorems in nonequilibrium statistical physics.

\begin{theorem}\label{Kurchan}
Let $c_1,c_2,\cdots,c_r$ be a family of cycles passing through a common state $i$. Let
\begin{equation}
h_t(\lambda_1,\cdots,\lambda_r) = E_ie^{t(\lambda_1K^{c_1}_t+\cdots+\lambda_rK^{c_r}_t)}.
\end{equation}
Then for each $t\geq 0$ and any $\lambda_1,\cdots,\lambda_r\in\mathbb{R}$,
\begin{equation}
\begin{split}
& h_t(\lambda_1,\cdots,\lambda_k,\cdots,\lambda_r) = h_t(\lambda_1,\cdots,-(\lambda_k+\rho^{c_k}),\cdots,\lambda_r).
\end{split}
\end{equation}
\end{theorem}

\begin{proof}
We only need to prove this theorem when $k = 1$. The proof of the other cases is totally the same. By Theorem \ref{transient}, we have
\begin{eqnarray*}
&& h_t(\lambda_1,\lambda_2,\cdots,\lambda_r) = E_ie^{t(\lambda_1K^{c_1}_t+\cdots+\lambda_rK^{c_r}_t)} \\
&=& \sum_{n_1,\cdots,n_r\in\mathbb{Z}}e^{\lambda_1n_1+\lambda_2n_2+\cdots+\lambda_rn_r} P_i(K^{c_1}_t=n_1/t,K^{c_2}_t=n_2/t,\cdots,K^{c_r}_t=n_r/t) \\
&=& \sum_{n_1,\cdots,n_r\in\mathbb{Z}}e^{\lambda_1n_1+\lambda_2n_2+\cdots+\lambda_rn_r}
P_i(K^{c_1}_t=-n_1/t,K^{c_2}_t=n_2/t,\cdots,K^{c_r}_t=n_r/t)e^{n_1\rho^{c_1}} \\
&=& \sum_{n_1,\cdots,n_r\in\mathbb{Z}}e^{(\lambda_1+\rho^{c_1})n_1+\lambda_2n_2+\cdots+\lambda_rn_r} P_i(K^{c_1}_t=-n_1/t,K^{c_2}_t=n_2/t,\cdots,K^{c_r}_t=n_r/t) \\
&=& E_ie^{t(-(\lambda_1+\rho^{c_1})K^{c_1}_t+\lambda_2K^{c_2}_t+\cdots+\lambda_rK^{c_r}_t})
= h_t(-(\lambda_1+\rho^{c_1}),\lambda_2,\cdots,\lambda_r),
\end{eqnarray*}
which gives the desired result.
\end{proof}

Theorems of the following type are called integral fluctuation theorems in nonequilibrium statistical physics.

\begin{theorem}
Let $c_1,c_2,\cdots,c_r$ be a family of cycles passing through a common state $i$. Then for each $t\geq 0$,
\begin{equation}
E_ie^{-t(K^{c_1}_t\rho^{c_1}+K^{c_2}_t\rho^{c_2}+\cdots+K^{c_r}_t\rho^{c_r})} = 1.
\end{equation}
\end{theorem}

\begin{proof}
By Theorem \ref{Kurchan}, for any $\lambda_1,\cdots,\lambda_r\in\mathbb{R}$,
\begin{equation}
E_ie^{t(\lambda_1K^{c_1}_t+\cdots+\lambda_rK^{c_r}_t)} = E_ie^{-t((\lambda_1+\rho^{c_1})K^{c_1}_t+\cdots+(\lambda_r+\rho^{c_r})K^{c_r}_t)}.
\end{equation}
If we take $\lambda_k=-\rho^{c_k}$ for each $k$ in the above equation, we obtain the desired result.
\end{proof}

The large deviations of the empirical net circulations and the symmetry of the rate function are stated in the following theorem. Theorems of the following type are called Gallavotti-Cohen-type fluctuation theorems in nonequilibrium statistical physics.

\begin{theorem}\label{LDPnetcirculation}
Let $c_1,c_2,\cdots,c_r$ be a family of cycles passing through a common state $i$. Then under $P_i$, the law of $(K^{c_1}_t,K^{c_2}_t,\cdots,K^{c_r}_t)$ satisfies a large deviation principle with rate $t$ and good rate function $I^{c_1,c_2,\cdots,c_r}_K:\mathbb{R}^r\rightarrow[0,\infty]$. Moreover, for each $1\leq k\leq r$, the rate function $I^{c_1,c_2,\cdots,c_r}_K$ has the following symmetry:
\begin{equation}
I^{c_1,c_2,\cdots,c_r}_K(x_1,\cdots,x_k,\cdots,x_r) = I^{c_1,c_2,\cdots,c_r}_K(x_1,\cdots,-x_k,\cdots,x_r)-\rho^{c_k}x_k.
\end{equation}
\end{theorem}

\begin{proof}
We only need to prove this theorem when $k = 1$. The proof of the other cases is totally the same. Let $F:\mathbb{R}^{2r}\rightarrow\mathbb{R}^r$ be a continuous map defined as
\begin{equation}
F(x_1,y_1,\cdots,x_r,y_r) = (x_1-y_1,\cdots,x_r-y_r).
\end{equation}
Then we have
\begin{equation}
F(J^{c_1}_t,J^{c_1-}_t,\cdots,J^{c_r}_t,J^{c_r-}_t) = (K^{c_1}_t,\cdots,K^{c_r}_t).
\end{equation}
By Corollary \ref{LDPconjugate}, the law of $(J^{c_1}_t,J^{c_1-}_t,\cdots,J^{c_r}_t,J^{c_r-}_t)$ satisfies a large deviation principle with rate $t$ and good rate function $I^{c_1,c_1-,\cdots,c_r,c_r-}$. Using the contraction principle, we see that the law of $(K^{c_1}_t,\cdots,K^{c_r}_t)$ satisfies a large deviation principle with rate $t$ and good rate function
\begin{equation}
I^{c_1,\cdots,c_r}_K(z_1,\cdots,z_r) = \inf_{x_1-y_1=z_1,\cdots,x_r-y_r=z_r}I^{c_1,c_1-,\cdots,c_r,c_r-}(x_1,y_1,\cdots,x_r,y_r).
\end{equation}
Thus we have
\begin{eqnarray*}
&& I^{c_1,\cdots,c_r}_K(z_1,z_2,\cdots,z_r) \\
&=& \inf_{x_1-y_1=z_1,\cdots,x_r-y_r=z_r}I^{c_1,c_1-,\cdots,c_r,c_r-}(x_1,y_1,x_2,y_2,\cdots,x_r,y_r) \\
&=& \inf_{x_1-y_1=z_1,\cdots,x_r-y_r=z_r}I^{c_1,c_1-,\cdots,c_r,c_r-}(y_1,x_1,x_2,y_2,\cdots,x_r,y_r)-
\rho^{c_1}(x_1-y_1) \\
&=& \inf_{y_1-x_1=-z_1,x_2-y_2=z_2,\cdots,x_r-y_r=z_r} I^{c_1,c_1-,\cdots,c_r,c_r-}(y_1,x_1,x_2,y_2,\cdots,x_r,y_r)-\rho^{c_1}z_1 \\
&=& I^{c_1,\cdots,c_r}_K(-z_1,z_2,\cdots,z_r)-\rho^{c_1}z_1,
\end{eqnarray*}
which gives the desired result.
\end{proof}

\subsection{Applications in biochemistry}\label{biochemistry}
One of the most important branch of biochemistry is enzyme kinetics, which studies chemical reactions that are catalyzed by enzymes. Recently, it has been made possible to study enzyme kinetics at the single-molecule level, in which case the concept of concentration makes no sense and the behavior of enzymes must be studied in a single-molecule way.

Let us consider the following three-step reversible Michaelis-Menten enzyme kinetics \cite{beard2008chemical, ge2008waiting, ge2012stochastic}:
\begin{equation}\label{Michaelis}
E+S\;\autorightleftharpoons{}{}\;ES\;\autorightleftharpoons{}{}\;EP\autorightleftharpoons{}{}\;E+P,
\end{equation}
where $E$ is an enzyme involved in converting the substrate $S$ into the product $P$. If there is only one enzyme molecule, it may transition stochastically among three states: the free enzyme $E$, the enzyme-substrate complex $ES$, and the enzyme-product complex $EP$. From the perspective of a single enzyme molecule, the Michaelis-Menten enzyme kinetics \eqref{Michaelis} can be modeled by the three-state Markov chain illustrated in Figure \ref{enzyme}(a). However, single-substrate enzymes are actually rather rare in biochemistry \cite{ge2012multivariable}. If the enzyme $E$ can catalyze multiple chemical reactions simultaneously with substrates $S_1,S_2,\cdots,S_n$ and products $P_1,P_2,\cdots,P_n$, then the topology of the Markov chain model will contain multiple cycles passing through a common state $E$, as illustrated in Figure \ref{enzyme}(b).
\begin{figure}[!htb]
\begin{center}
\centerline{\includegraphics[width=0.8\textwidth]{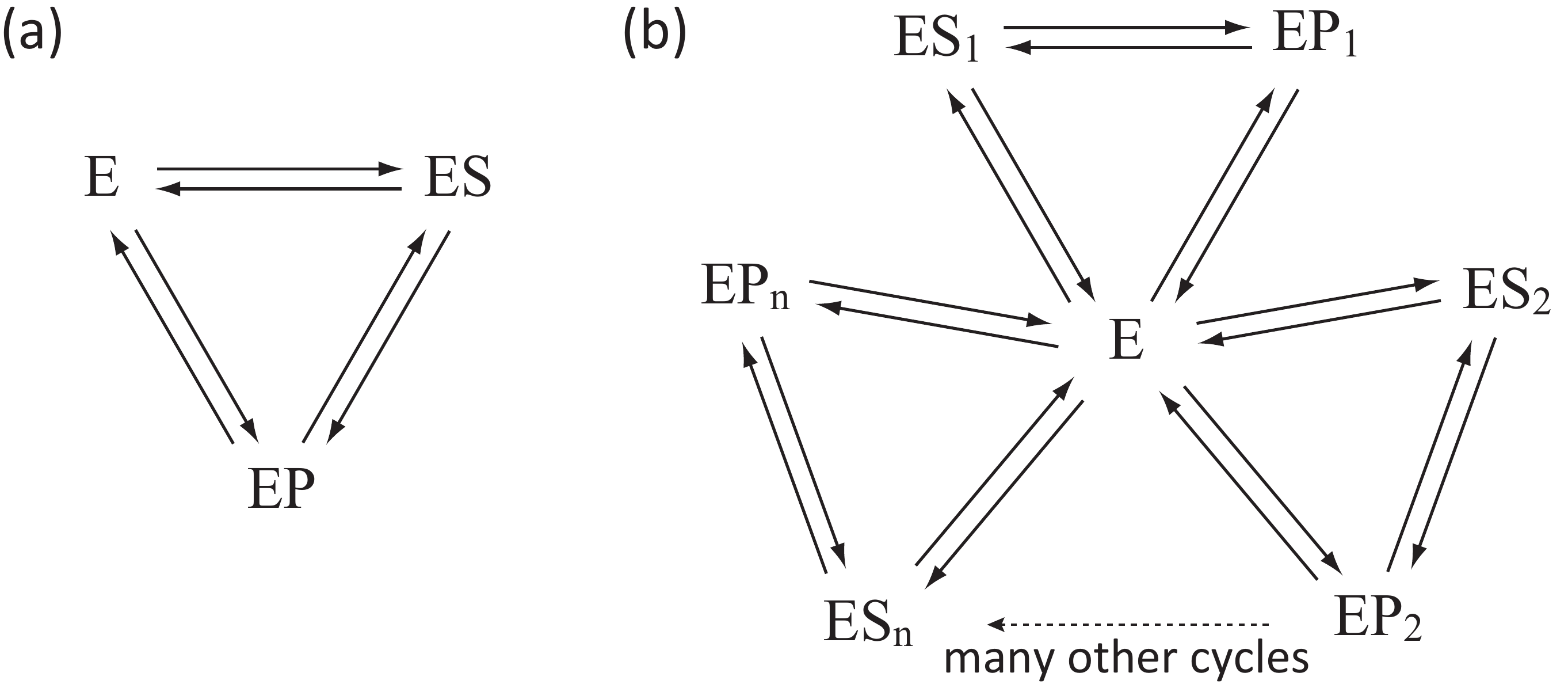}}
\caption{\textbf{Markov chain models of enzyme kinetics.} (a) The Markov chain model of single-substrate enzyme kinetics. (b) The Markov chain model of multiple-substrate enzyme kinetics.}\label{enzyme}
\end{center}
\end{figure}

We assume that the Markov chain illustrated in Figure \ref{enzyme}(b) starts from state $E$. If the Markov chain forms a clockwise cycle $c_k=(E,ES_k,EP_k)$, then the substrate $S_k$ is converted into the product $P_k$ for one time. Similarly, if the Markov chain forms a counterclockwise cycle $c_k-=(E,EP_k,ES_k)$, then the product $P_k$ is converted into the substrate $S_k$ for one time. Thus the empirical net circulation $K^{c_k}_t$ of cycle $c_k$ represents the net number of conversions from the substrate $S_k$ into the product $P_k$ per unit time and the quantity
\begin{equation}
W^{c_k}_t = K^{c_k}_t\rho^{c_k}
\end{equation}
represents the fluctuating chemical work done along cycle $c_k$ \cite{qian2006generalized, ge2012stochastic, ge2012multivariable}, where $\rho^{c_k}$ is the affinity of cycle $c_k$ (see Definition \ref{affinity}). In fact, the results of this paper can be directly applied to establish the multivariate fluctuation theorems of the empirical net circulations of cycles $c_1,c_2,\cdots,c_n$ and the fluctuating chemical works done along cycles $c_1,c_2,\cdots,c_n$. This shows that our work would have a board application prospect in biochemistry.

\section*{Acknowledgements}
The authors gratefully acknowledge Prof. Hao Ge at Peking University for stimulating discussions and gratefully acknowledge financial support from NSFC 11271029 and NSFC 11171024. The first author also acknowledges financial support from the Academic Award for Young Ph.D. Researchers granted by the Ministry of Education of China.

%%%%%%%%%% 参考文献格式 %%%%%%%%%%
\setlength{\bibsep}{5pt}
\small\bibliographystyle{unsrt}
\bibliography{Haldane}

\begin{thebibliography}{10}

\bibitem{qian1979decomposition}
MP~Qian and Min Qian.
\newblock The decomposition into a detailed balance part and a circulation part
  of an irreversible stationary markov chain.
\newblock {\em Scientia Sinica, Special Issue (II)}, 69, 1979.

\bibitem{qian1981markov}
C~Qian, Min Qian, and MP~Qian.
\newblock Markov chain as a model of hill's theory on circulation.
\newblock {\em Scientia Sinica}, 24(10):1431--1448, 1981.

\bibitem{minping1982circulation}
Qian Minping, Qian Min, and Qian Cheng.
\newblock Circulation distribution of a markov chain.
\newblock {\em Scientia Sinica A}, 25:31--40, 1982.

\bibitem{qian1982circulation}
Qian Minping and Qian Min.
\newblock Circulation for recurrent markov chains.
\newblock {\em Probability Theory and Related Fields}, 59(2):203--210, 1982.

\bibitem{qian1984circulations}
MP~Qian, M~Qian, and C~Qian.
\newblock Circulations of markov-chains with continuous-time and the
  probability interpretation of some determinants.
\newblock {\em SCIENTIA SINICA SERIES A-MATHEMATICAL PHYSICAL ASTRONOMICAL \&
  TECHNICAL SCIENCES}, 27(5):470--481, 1984.

\bibitem{kalpazidou1990asymptotic}
S~Kalpazidou.
\newblock Asymptotic behaviour of sample weighted circuits representing
  recurrent markov chains.
\newblock {\em Journal of applied probability}, pages 545--556, 1990.

\bibitem{qian1991reversibility}
MP~Qian, Min Qian, and GL~Gong.
\newblock The reversibility and the entropy production of markov processes.
\newblock {\em Contemp. Math}, 118:255--261, 1991.

\bibitem{jiang2004mathematical}
Da-Quan Jiang and Min Qian.
\newblock {\em Mathematical theory of nonequilibrium steady states: on the
  frontier of probability and dynamical systems}.
\newblock Number 1833. Springer, 2004.

\bibitem{kalpazidou2007cycle}
Sophia~L Kalpazidou.
\newblock {\em Cycle representations of Markov processes}, volume~28.
\newblock Springer, 2007.

\bibitem{zhang2012stochastic}
Xue-Juan Zhang, Hong Qian, and Min Qian.
\newblock Stochastic theory of nonequilibrium steady states and its
  applications. part i.
\newblock {\em Physics Reports}, 510(1):1--86, 2012.

\bibitem{ge2012stochastic}
Hao Ge, Min Qian, and Hong Qian.
\newblock Stochastic theory of nonequilibrium steady states. part ii:
  Applications in chemical biophysics.
\newblock {\em Physics Reports}, 510(3):87--118, 2012.

\bibitem{kolmogoroff1936theorie}
Andrei Kolmogoroff.
\newblock Zur theorie der markoffschen ketten.
\newblock {\em Mathematische Annalen}, 112(1):155--160, 1936.

\bibitem{hill2012free}
Terrell Hill.
\newblock {\em Free energy transduction in biology: the steady-state kinetic
  and thermodynamic formalism}.
\newblock Elsevier, 2012.

\bibitem{hill2013free}
Terrell~L Hill.
\newblock {\em Free Energy Transduction and Biochemical Cycle Kinetics}.
\newblock Courier Dover Publications, 2013.

\bibitem{qian2006generalized}
Hong Qian and X~Sunney Xie.
\newblock Generalized haldane equation and fluctuation theorem in the
  steady-state cycle kinetics of single enzymes.
\newblock {\em Physical Review E}, 74(1):010902, 2006.

\bibitem{ge2008waiting}
Hao Ge.
\newblock Waiting cycle times and generalized haldane equality in the
  steady-state cycle kinetics of single enzymes.
\newblock {\em The Journal of Physical Chemistry B}, 112(1):61--70, 2008.

\bibitem{ge2012multivariable}
Hao Ge.
\newblock Multivariable fluctuation theorems in the steady-state cycle kinetics
  of single enzyme with competing substrates.
\newblock {\em Journal of Physics A: Mathematical and Theoretical},
  45(21):215002, 2012.

\bibitem{seifert2012stochastic}
Udo Seifert.
\newblock Stochastic thermodynamics, fluctuation theorems and molecular
  machines.
\newblock {\em Reports on Progress in Physics}, 75(12):126001, 2012.

\bibitem{evans1993probability}
Denis~J Evans, EGD Cohen, and GP~Morriss.
\newblock Probability of second law violations in shearing steady states.
\newblock {\em Physical Review Letters}, 71(15):2401, 1993.

\bibitem{gallavotti1995dynamical}
Giovanni Gallavotti and EGD Cohen.
\newblock Dynamical ensembles in stationary states.
\newblock {\em Journal of Statistical Physics}, 80(5-6):931--970, 1995.

\bibitem{lebowitz1999gallavotti}
Joel~L Lebowitz and Herbert Spohn.
\newblock A gallavotti--cohen-type symmetry in the large deviation functional
  for stochastic dynamics.
\newblock {\em Journal of Statistical Physics}, 95(1-2):333--365, 1999.

\bibitem{jiang2003entropy}
Da-Quan Jiang, Min Qian, and Fu-Xi Zhang.
\newblock Entropy production fluctuations of finite markov chains.
\newblock {\em Journal of Mathematical Physics}, 44(9):4176--4188, 2003.

\bibitem{seifert2005entropy}
Udo Seifert.
\newblock Entropy production along a stochastic trajectory and an integral
  fluctuation theorem.
\newblock {\em Physical review letters}, 95(4):040602, 2005.

\bibitem{chung1967markov}
Kai~Lai Chung.
\newblock {\em Markov chains}.
\newblock Springer, 1967.

\bibitem{mariani2012large}
Mauro Mariani, Yuhao Shen, and Lorenzo Zambotti.
\newblock Large deviations for the empirical measure of markov renewal
  processes.
\newblock {\em arXiv preprint arXiv:1203.5930}, 2012.

\bibitem{varadhan1984large}
SR~Srinivasa Varadhan, SR~Srinivasa Varadhan, and SR~Srinivasa Varadhan.
\newblock {\em Large deviations and applications}, volume~46.
\newblock SIAM, 1984.

\bibitem{dembo1998large}
Amir Dembo and Ofer Zeitouni.
\newblock {\em Large deviations techniques and applications}, volume~2.
\newblock Springer, 1998.

\bibitem{borwein2010convex}
Jonathan~M Borwein and Adrian~S Lewis.
\newblock {\em Convex analysis and nonlinear optimization: theory and
  examples}, volume~3.
\newblock Springer, 2010.

\bibitem{beard2008chemical}
Daniel~A Beard and Hong Qian.
\newblock {\em Chemical biophysics: quantitative analysis of cellular systems}.
\newblock Cambridge University Press, 2008.

\end{thebibliography}

\end{document}